\documentclass[12pt, reqno]{amsart}
\usepackage{amsmath,amsfonts,amsbsy,amsgen,amscd,mathrsfs,amssymb,amsthm}
\usepackage[usenames,dvipsnames]{xcolor}
\usepackage[colorlinks=true,citecolor=red,linkcolor=blue]{hyperref}
\usepackage[a4paper,margin=1in]{geometry}
\usepackage{setspace}
\usepackage{amsmath}
\usepackage{bm}
\usepackage{color}
\allowdisplaybreaks[4]
\numberwithin{equation}{section}
\newtheorem{theorem}{Theorem}[section]
\newtheorem{definition}[theorem]{Definition}
\newtheorem{lemma}[theorem]{Lemma}
\newtheorem{corollary}[theorem]{Corollary}
\newtheorem{remark}[theorem]{Remark}
\newtheorem{proposition}[theorem]{Proposition}

\newtheorem{problem}[theorem]{Problem}

\title[the dual Minkowski Problem for $q$-torsional rigidity]{the dual Minkowski Problem for $q$-torsional rigidity}

\author{Xia Zhao and Peibiao Zhao}


\begin{document}
\begin{abstract}
The Minkowski problem for torsional rigidity ($2$-torsional rigidity) was firstly studied by Colesanti and Fimiani \cite{CA} using variational method. Moreover, Hu \cite{HJ00} also studied this problem by the method of curvature flows and obtained the existence of smooth even solutions. In addition, the smooth non-even solutions to the Orlicz Minkowski problem $w. r. t$ $q$-torsional rigidity were given by Zhao et al. \cite{ZX} through a Gauss curvature flow.

The dual curvature measure and the dual Minkowski problem were first posed and considered by Huang, Lutwak, Yang and Zhang in \cite{HY}. The dual Minkowski problem is a very important problem, which has greatly contributed to the development of the dual Brunn-Minkowski theory and extended the other types dual Minkowski problem.

To the best of our knowledge, the dual Minkowski problem $w. r. t$ ($q$) torsional rigidity is still open because the dual ($q$) torsional measure is blank. Thus, it is a natural problem to consider the dual Minkowski problem for ($q$) torsional rigidity. In this paper, we introduce the $p$-th dual $q$-torsional measure and propose the $p$-th dual Minkowski problem for $q$-torsional rigidity with $q>1$. Then we confirm the existence of smooth even solutions for $p<n$ ($p\neq 0$) to the $p$-th dual Minkowski problem for $q$-torsional rigidity by method of a Gauss curvature flow. Specially, we also obtain the smooth non-even solutions with $p<0$ to this problem.
\end{abstract}
\maketitle
\vskip 20pt
\section{Introduction and main results}

The classical Minkowski problem argues the existence, uniqueness and regularity of a convex body whose surface area measure is equal to a pre-given Borel measure $\mu$ on the sphere $S^{n-1}\subset\mathbb{R}^n$ in the Brunn-Minkowski theory. If the given measure has a positive continuous density, the Minkowski problem can be seen as the problem of prescribing the Gauss curvature in differential geometry.

The Minkowski problem plays an important role in the study of convex geometry, and the research of Minkowski problem has also promoted the development of fully nonlinear partial differential equations. In addition, the Minkowski problem has produced some variations of it, among which the $L_p$ Minkowski problem is particularly important because the $L_p$ Minkowski problem contains some special versions. Namely: when $p=1$, it is the classical Minkowski problem; when $p=0$, it is the famous log-Minkowski problem \cite{BO}; when $p=-n$, it is the centro-affine Minkowski problem \cite{ZG}. The $L_p$ Minkowski problem with $p>1$ was first proposed and studied by Lutwak \cite{LE0}, whose solution plays a key role in establishing the $L_p$ affine Sobolev inequality \cite{HC2, LE01}. Fortunately, Haberl, Lutwak, Yang and Zhang \cite{HC1} proposed and studied the even Orlicz Minkowski problem in 2010 which is a more generalized Minkowski type problem, and its result contains the classical Minkowski problem and the $L_p$ Minkowski problem.

The ($L_p$) dual Brunn-Minkowski theory was introduced by Lutwak in \cite{LE, LE1}, replacing the convex bodies and their ($L_p$) Minkowski addition by the star bodies and their ($L_p$) radial addition. Many notions in the ($L_p$) Brunn-Minkowski theory for convex bodies have their dual analogues in the ($L_p$) dual theory \cite{LE, LE1} or the book \cite{GR02} written by Gardner for more background and references. Successfully, some corresponding inequalities and problems were solved in the ($L_p$) dual Brunn-Minkowski theory. For example, the famous Busemann-Petty problems were solved in \cite{GR, GR0, GR01, ZGY} etc.. But, what acts as the dual counterparts of the geometric measures in the Brunn-Minkowski theory are not clear, behind this lay our inability to calculate the differentials of dual quermassintegrals. Until the work of Huang, Lutwak, Yang and Zhang \cite{HY} in 2016, they established the variational formulas for dual quermassintegrals with the convex hulls instead of the Wullf shapes, this is not only a completely different approach, but Aleksandrov’s variational principle was established without using the Minkowski mixed-volume inequality. Therefrom, the $q$-th dual curvature measure $\widetilde{C}_q(\Omega,\cdot)$ was discovered and posed the dual Minkowski problem below: Given a finite Borel measure $\mu$ on $S^{n-1}$ and $q\in\mathbb{R}$, find the necessary and sufficient condition(s) on $\mu$ so that there exists a convex body $\Omega$ containing the origin in its interior and $\mu(\cdot)=\widetilde{C}_q(\Omega,\cdot)$. This remarkable work \cite{HY} not only promotes the development of the dual Bruun-Minkowski theory, but also extends the other types dual Minkowski problems. Huang et al. \cite{HY} provided the existence of solutions (i.e., origin-symmetric convex bodies) to the dual Minkowski problem with $q\in(0,n)$ for even measure $\mu$. For $q<0$, the existence and uniqueness of solutions to the dual Minkowski problem were given recently in \cite{ZY} by Zhao. The dual Orlicz curvature measure and its Minkowski problem were established by Zhu, Xing and Ye \cite{ZB}. Two special cases of the dual Minkowski problem include the logarithmic Minkowski problem for $q=n$ and the Aleksandrov problem when $q=0$. More works with respect to the dual Minkowski problem, one can see \cite{GR1, GR2, HY2, LR, ZY, ZY1} for details.

With the development of the Minkowski problems and their dual analogues, the study of this types problems have had profound influence and inspired many other problems with similar natures. For example, some geometric measures with physical backgrounds have been introduced into the Brunn-Minkowski theory, naturally, the related Minkowski type problems have also been posed and gradually studied. Early on, Jerison \cite{JE} introduced the capacity Minkowski problem and studied this problem through the prescribing capacity curvature measure. Further, Xiao \cite{XJ} prescribed the capacitary curvature measures on planar convex domains. Moreover, Colesanti, Nystr\"{o}m, Salani, Xiao, Yang and Zhang \cite{C0} established the Hadamard variational formula and considered the Minkowski problem for $p$-capacity. Very recently, based on the definition of $p$-capacity and the dual Minkowski problem, Ji \cite{JLW} introduced the $q$-th dual $p$-capacity measure and considered the existence of solutions to the $p$-capacity dual Minkowski problem when $1<p<n$ and $q<0$.

In addition, the Minkowski problem for torsional rigidity ($2$-torsional rigidity) with physical backgrounds was firstly studied by Colesanti and Fimiani \cite{CA} using variational method. The Minkowski problem for $2$-torsional rigidity was extended to the $L_p$ version by Chen and Dai \cite{CZ} who proved the existence of solutions for any fixed $p>1$ and $p\neq n+2$, Hu and Liu \cite{HJ01} for $0 < p < 1$. The Orlicz Minkowski problem $w.r.t.$ $2$-torsional rigidity was first developed and proven by Li and Zhu \cite{LN}. Further, the $L_p$ variational formula for $q$-torsional rigidity with $q>1$ was established by Huang, Song and Xu \cite{HY0}. The authors \cite{ZX} of this article have also had a systematic investigation on this topic and  proposed the Orlicz Minkowski problem for $q$-torsional rigidity with $q>1$ and obtained its smooth non-even solutions by a Gauss curvature flow.

Based on the foregoing  works for ($q$) torsional rigidity, we found that some beautiful conclusions with respect to the Minkowski problem for ($q$) torsional rigidity have been obtained in the Brunn-Minkowski theory. But, the counterpart of the Minkowski problem for $(q)$ torsional rigidity in the dual Brunn-Minkowski theory has not been considered, mainly because of lack of the corresponding dual ($q$) torsional measure. Thus, it is a very natural, important and challenging problem to introduce the dual ($q$) torsional measure, propose and solve the dual Minkowski problem for ($q$) torsional rigidity. It is believed that this research will contribute to the enrichment and development of the dual Brunn-Minkowski theory.

For convenience, we recall and state firstly the concept of $q$-torsional rigidity and its related contents as below. Let $\mathcal{K}^n$ be the collection of convex bodies in Euclidean space $\mathbb{R}^n$. The set of convex bodies containing the origin in their interiors in $\mathbb{R}^n$, we write $\mathcal{K}^n_o$. Moreover,  let $C^2_{+}$ be the class of convex bodies of $C^2$ if its boundary has the positive Gauss curvature.

We now do the needful. Let $\Omega\in\mathcal{K}^n$, the $q$-torsional rigidity $T_q(\Omega)$ \cite{CA1} with $q>1$ is defined by
\begin{align}\label{eq101}
\frac{1}{T_q(\Omega)}=\inf\bigg\{\frac{\int_{\Omega}|\nabla U|^q dy}{[\int_{\Omega}|U|dy]^q}:U\in W_0^{1,q}(\Omega),\int_{\Omega}|U|dy>0\bigg\}.
\end{align}
It is illustrated in \cite{BM,HH} that the functional defined in (\ref{eq101}) has a unique minimizer $u\in W_0^{1,q}(\Omega)$ satisfying the following boundary value problem
\begin{align}\label{eq102}
\left\{
\begin{array}{lc}
\Delta_qu=-1\ \ \  \text{in} \ \ \ \Omega,\\
u=0,\ \ \ \ \ \ \ \ \text{on} \ \ \ \ \partial\Omega,\\
\end{array}
\right.
\end{align}
where $$\Delta_qu={\rm div}(|\nabla u|^{q-2}\nabla u)$$
is the $q$-Laplace operator.

When $q=2$, it is the so-called torsional rigidity $T(\Omega)$ (or $2$-torsional rigidity $T_2(\Omega)$) of a convex body $\Omega$.

Applying (\ref{eq102}) and the Gauss-Green formula, we have

\begin{align}\label{eq103}
\int_{\Omega}|\nabla u|^qdy=\int_{\Omega}udy,
\end{align}
from (\ref{eq101}) and (\ref{eq103}), it follows
\begin{align}\label{eq104}
T_q(\Omega)=\frac{(\int_{\Omega}udy)^q}{\int_{\Omega}|\nabla u|^qdy}=\bigg(\int_{\Omega}udy\bigg)^{q-1}=\bigg(\int_{\Omega}|\nabla u|^qdy\bigg)^{q-1}.
\end{align}

With the aid of Pohozaev-type identities \cite{PU}, the $q$-torsional rigidity formula (\ref{eq104}) can be given by
\begin{align}\label{eq105}
T_q(\Omega)^{\frac{1}{q-1}}=&\frac{q-1}{q+n(q-1)}\int_{S^{n-1}}h(\Omega,\xi)d\mu^{tor}_{q}(\Omega,\xi)\\
\nonumber=&\frac{q-1}{q+n(q-1)}\int_{S^{n-1}}h(\Omega,\xi)|\nabla u|^qdS(\Omega,\xi).
\end{align}
Denoting $\frac{q-1}{q+n(q-1)}=\frac{b}{a}$ and $\widetilde{T}_q(\Omega)=T_q(\Omega)^{\frac{1}{q-1}}$, the $q$-torsional measure $\mu^{tor}_{q}(\Omega,\eta)$ is defined by
\begin{align}\label{eq106}
\mu^{tor}_q(\Omega,\eta)=\int_{\nu^{-1}(\eta)}|\nabla u(y)|^qd\mathcal{H}^{n-1}(y)=\int_\eta|\nabla u(\nu^{-1}(x))|^qdS(\Omega,x),
\end{align}
for any Borel set $\eta\subseteq S^{n-1}$. Here, $\nu:\partial\Omega\rightarrow S^{n-1}$ is the Gauss map, and $\mathcal{H}^{n-1}$ is the $(n-1)$-dimensional Hausdorff measure.

Motivated by the dual curvature measure and the dual Minkowski problem in \cite{HY} and the foregoing works with respect to ($q$) torsional rigidity. In the present paper, we focus on considering the $p$-th dual Minkowski problem for $q$-torsional rigidity with $q>1$ in the dual Brunn-Minkowski theory. Firstly, we give the definition of $p$-th dual $q$-torsional measure.
\begin{definition}\label{def11} Let $\Omega\in\mathcal{K}^n_o$, $q>1$ and $p\in\mathbb{R}$. We define the $p$-th dual $q$-torsional measure by
\begin{align*}
\widetilde{Q}_{q,n-p}(\Omega,\eta)=\frac{b}{a}\int_{\alpha_\Omega^*(\eta)}\rho(\Omega,v)^p|\nabla u(r(\Omega,v))|^qdv,
\end{align*}
for each Borel $\eta\subset S^{n-1}$ and $r(\Omega,v)=\rho(\Omega,v)v$, where $\frac{b}{a}=\frac{q-1}{n(q-1)+q}$, $\rho(\Omega,\cdot)$ is the radial function of $\Omega$, $\alpha_\Omega^*$ is the reverse radial Gauss image on $S^{n-1}$ and $dv$ is the spherical measure on $S^{n-1}$ (see Sect.$2$ for definitions and notations).
\end{definition}

The  $p$-th dual $q$-torsional rigidity $\widetilde{Q}_{q,n-p}(\Omega)$ of $\Omega\in\mathcal{K}_o^n$ is denoted by 
\begin{align*}
\widetilde{Q}_{q,n-p}(\Omega)=\widetilde{Q}_{q,n-p}(\Omega,S^{n-1})=\frac{b}{a}\int_{S^{n-1}}\rho(\Omega,v)^p|\nabla u(r(\Omega,v))|^qdv.
\end{align*}

Naturally, the $p$-th dual Minkowski problem for $q$-torsional rigidity can be proposed as follows.
\begin{problem}\label{pro12}
Let $q>1$ and $p\in\mathbb{R}$. Given a nonzero finite Borel measure $\mu$ on $S^{n-1}$, what are the necessary and sufficient conditions on $\mu$ such that a convex body $\Omega\in\mathcal{K}_o^n$ whose $p$-th dual $q$-torsional measure $\widetilde{Q}_{q,n-p}(\Omega,\cdot)$ is equal to the given measure $\mu$?
\end{problem}

If the given measure $\mu$ is absolutely continuous with respect to the Lebesgue measure and $\mu$ has a smooth density function $f:S^{n-1}\rightarrow (0,\infty)$, then, solving Problem \ref{pro12} can be equivalently viewed as solving the following Monge-Amp\`{e}re equation on $S^{n-1}$ (see (\ref{eq311}) for details):
\begin{align*}
f(v)=\frac{b}{a}|\nabla h_{\Omega}(v)+h(v)v|^{p-n}h_{\Omega}(v)|\nabla u(\nu^{-1}_\Omega(v))|^q\det(h_{ij}(v)+h_\Omega(v)\delta_{ij}),
\end{align*}
equivalently,
\begin{align}\label{eq107}
f(v)=\frac{b}{a}(|\nabla h|^2+h^2)^{\frac{p-n}{2}}h_{\Omega}(v)|\nabla u(\nu^{-1}_\Omega(v))|^q\det(h_{ij}(v)+h_\Omega(v)\delta_{ij}).
\end{align}
Here, $h$ is the unknown function on $S^{n-1}$ to be found, $\nabla h$ and $h_{ij}$ denote the gradient vector and the Hessian matrix of $h$ with respect to an orthonormal frame on $S^{n-1}$, and $\delta_{ij}$ is the Kronecker delta.

If the factor
\begin{align*}
(|\nabla h|^2+h^2)^{\frac{p-n}{2}}h_{\Omega}(v)
\end{align*}
is omitted in Equation (\ref{eq107}), then (\ref{eq107}) will become the partial differential equation of
the Minkowski problem for $q$-torsional rigidity. If only the factor $(|\nabla h|^2+h^2)^{\frac{p-n}{2}}$ is omitted, then Equation (\ref{eq107}) becomes the partial differential equation associated with the logarithmic Minkowski problem to $q$-torsional rigidity. The gradient component in (\ref{eq107}) significantly increases the difficulty of the problem when it is compared to the Minkowski problem or the logarithmic Minkowski problem for $q$-torsional rigidity.

In the present paper, we will investigate smooth solutions to the normalized $p$-th dual Minkowski problem for $q$-torsional rigidity by method of a Gauss curvature flow. The normalized equation is
\begin{align}\label{eq108}
\frac{\frac{a}{b}\widetilde{Q}_{q,n-p}(\Omega)}{\int_{S^{n-1}}f(v)dv}f(v)=(|\nabla h|^2+h^2)^{\frac{p-n}{2}}h_{\Omega}(v)|\nabla u(\nu^{-1}_\Omega(v))|^q\det(h_{ij}(v)+h_\Omega(v)\delta_{ij}).
\end{align}
By homogeneity, it is clear to see that if $h(x)$ is a solution of (\ref{eq108}), then $\frac{a}{b}\bigg[\frac{\frac{a}{b}\widetilde{Q}_{q,n-p}(\Omega)}{\int_{S^{n-1}}f(x)dx}\bigg]^{-p}h$ is a solution of (\ref{eq107}).
The Gauss curvature flow is particularly essential, which was first introduced and studied by Firey \cite{FI} to model the shape change of worn stones. It can mainly be used to study the existence of smooth solutions to the famous Minkowski (type) problems. For examples, Chen, Huang and Zhao \cite{CC} obtained the existence of smooth even solutions to the $L_p$ dual Minkowski problem by a Gauss curvature flow. Liu and Lu \cite{LY} solved the dual Orlicz-Minkowski problem and obtained the existence of smooth solutions by a Gauss curvature flow. Various Gauss curvature flows have been extensively studied, see \cite{AB, BR, BR1, CK0, HJ00, HJ, LY1, ZX0} and the references therein.

In this article, we consider the $p$-th dual Minkowski problem for $q$-torsional rigidity with $q>1$. We confirm the existence of smooth even solutions for $p<n$ ($p\neq 0$) to the $p$-th dual Minkowski problem for $q$-torsional rigidity by method of a Gauss curvature flow. Specially, we also obtain the smooth non-even solutions with $p<0$ for this problem.

Let $\partial\Omega_0$ be a smooth, closed and origin-symmetric strictly convex hypersurface in $\mathbb{R}^n$ for $p<n$ ($p\neq 0$) and $f$ is a positive smooth even function on $S^{n-1}$. Specially, for $p<0$, let $\partial\Omega_0$ be a smooth, closed and strictly convex hypersurface in $\mathbb{R}^n$ containing the origin in its interior and $f$ is a positive smooth function on $S^{n-1}$. We construct and consider the long-time existence and convergence of a following Gauss curvature flow which is a family of convex hypersurfaces $\partial\Omega_t$  parameterized by  smooth maps $X(\cdot ,t):
S^{n-1}\times (0, \infty)\rightarrow \mathbb{R}^n$
satisfying the initial value problem
\begin{align}\label{eq109}
\left\{
\begin{array}{lc}
\frac{\partial X(x,t)}{\partial t}=-\lambda(t)f(v)\frac{(|\nabla h|^2+h^2)^{\frac{n-p}{2}}}{|\nabla u(X,t)|^q}\mathcal{K}
(x,t)v+X(x,t),  \\
X(x,0)=X_0(x),\\
\end{array}
\right.
\end{align}
where $\mathcal{K}(x,t)$ is the Gauss curvature of hypersurface $\partial\Omega_t$,  $v=x$ is the
outer unit normal at $X(x,t)$, $X\cdot v $ represents the standard inner product of $X$ and $v$, and $\lambda(t)$ is defined as follows
\begin{align*}\lambda(t)=\frac{\int_{S^{n-1}}\rho^p|\nabla u|^qdv}{\int_{S^{n-1}}f(x)dx}=\frac{\frac{a}{b}\widetilde{Q}_{q,n-p}(\Omega_t)}{\int_{S^{n-1}}f(x)dx},\end{align*}
where $\rho$ is the radial function of convex body $\Omega_t$ which enclosed by convex hypersurface $\partial \Omega_t$.

In order to discuss conveniently the flow (\ref{eq109}), we introduce the functional for $p\neq 0$ as follows:
\begin{align}\label{eq110}
\Phi(\Omega_t)=\frac{\int_{S^{n-1}}\log h(x,t)f(x)dx}{\int_{S^{n-1}}f(x)dx}-\log\bigg(\int_{S^{n-1}}\rho^p(v,t)|\nabla u|^qdv\bigg)^\frac{1}{p(1+q)}.
\end{align}
Here, $h(x,t)$ and $\rho(v,t)$ are the support function and the radial function of convex body $\Omega_t$, respectively.

Combining Equation (\ref{eq108}) with the flow (\ref{eq109}), we establish the following result in this article.

\begin{theorem}\label{thm13}
Let $q>1$, $p<n$ ($p\neq 0$) and $\partial\Omega_0$ be a smooth, closed and origin-symmetric strictly convex hypersurface in $\mathbb{R}^n$, $f$ is a positive smooth even function on $S^{n-1}$. Then, the flow (\ref{eq109}) has a unique smooth even convex solution $\partial\Omega_t=X(S^{n-1},t)$. Moreover, when $t\rightarrow\infty$, there is a subsequence of $\partial\Omega_t$ converges in $C^{\infty}$ to a smooth, closed, origin-symmetric strictly convex hypersurface $\partial\Omega_\infty$, the support function of convex body $\Omega_\infty$ enclosed by $\partial\Omega_\infty$ satisfies (\ref{eq108}).

Specially, for $p<0$, let $\partial\Omega_0$ be a smooth, closed and strictly convex hypersurface in $\mathbb{R}^n$ containing the origin in its interior,
$f$ is a positive smooth function on $S^{n-1}$. Then, the flow (\ref{eq109}) has a unique smooth non-even convex solution $\partial\Omega_t=X(S^{n-1},t)$. Moreover, when $t\rightarrow\infty$, there is a subsequence of $\partial\Omega_t$ converges in $C^{\infty}$ to a smooth, closed and strictly convex hypersurface $\partial\Omega_\infty$, the support function of convex body $\Omega_\infty$ enclosed by $\partial\Omega_\infty$ satisfies (\ref{eq108}).
\end{theorem}

This paper is organized as follows. We collect some necessary background materials in Section \ref{sec2}. In Section \ref{sec3}, we introduce the $p$-th dual $q$-torsional measure and $p$-th dual $q$-torsional rigidity. Moreover, some properties of $p$-th dual $q$-torsional measure and the variational formula of $p$-th dual $q$-torsional rigidity are given. In Section \ref{sec4}, we give the scalar form  of flow (\ref{eq109}) by the support function and discuss the monotonicity of two important functionals along the flow (\ref{eq109}). In Section \ref{sec5}, we give the priori estimates for solutions to the flow (\ref{eq109}). We obtain the convergence of flow (\ref{eq109}) and complete the proof of Theorem \ref{thm13} in Section \ref{sec6}.

\section{\bf Preliminaries}\label{sec2}
In this subsection, we give a brief review of some relevant notions and terminologies required for this article. One can refer to \cite{UR}, \cite{HY} and a book of Schneider \cite{SC} for details.
\subsection{Convex bodies and star bodies} Let $\mathbb{R}^n$ be the $n$-dimensional Euclidean space. The origin-centered unit ball $\{y\in\mathbb{R}^n:|y|\leq 1\}$ is always denoted by $\mathcal{B}$, and its boundary by $S^{n-1}$. Write $\omega_n$ for the volume of $\mathcal{B}$ and recall that its surface area is $n\omega_n$.

Let $\partial\Omega$ be a smooth, closed and strictly convex hypersurface containing the origin in its interior. The support function of a convex body $\Omega$ enclosed by $\partial\Omega$ is defined by
\begin{align*}h_\Omega(\xi)=h(\Omega,\xi)=\max\{\xi\cdot y:y\in\Omega\},\quad \forall\xi\in S^{n-1},\end{align*}
and the radial function of $\Omega$ with respect to $o$ (origin) $\in\mathbb{R}$ is defined by
\begin{align*}\rho_{\Omega}(v)=\rho(\Omega,v)=\max\{c>0:cv\in\Omega\},\quad  v\in S^{n-1}.\end{align*}
We easily obtain that the support function is homogeneous of degree $1$ and the radial function is homogeneous of degree $-1$.

For $\Omega\in\mathcal{K}_o^n$, its polar body $\Omega^*\in\mathcal{K}_o^n$ is defined by
\begin{align*}
\Omega^*=\{z\in\mathbb{R}^n:z\cdot y\leq 1, \text{for all}~~ y\in \Omega\}.
\end{align*}
It is clear that
\begin{align*}
\rho_\Omega=\frac{1}{h_{\Omega^*}} \quad \text{and} \quad h_\Omega=\frac{1}{\rho_{\Omega^*}}.
\end{align*}

The Minkowski combination, $a\Omega_1+b\Omega_2\in\mathcal{K}^n$, is defined as, for $\Omega_1,\Omega_2\in\mathcal{K}^n$ and $a,b>0$,
\begin{align*}
a\Omega_1+b\Omega_2={ay+bz:y\in \Omega_1, z\in \Omega_2},
\end{align*}
and
\begin{align*}
h(a\Omega_1+b\Omega_2,\cdot)=ah(\Omega_1,\cdot)+bh(\Omega_2,\cdot).
\end{align*}

For $a,b>0$, $q\neq 0$, the $L_q$-combination of $\Omega_1,\Omega_2\in\mathcal{K}_o^n$ is defined by
\begin{align*}
a\cdot \Omega_1+_qb\cdot \Omega_2=\bigcap_{v\in S^{n-1}}\bigg\{y\in \mathbb{R}^n:y\cdot v\leq [ah(\Omega_1,v)^q+bh(\Omega_2,v)^q]^{\frac{1}{q}}\bigg\}.
\end{align*}
If $q\geq 1$, $[ah(\Omega_1,\cdot)^q+bh(\Omega_2,\cdot)^q]^{\frac{1}{q}}$ is the support function of a convex body, but it is not necessary when $q<1$. For $q=0$,
\begin{align*}
a\cdot \Omega_1+_0b\cdot \Omega_2=\bigcap_{v\in S^{n-1}}\bigg\{y\in \mathbb{R}^n:y\cdot v\leq h(\Omega_1,v)^ah(\Omega_2,v)^b\bigg\}.
\end{align*}

Let $\Omega_1, \Omega_2\in\mathbb{R}^n$ be compact and star-shaped (with respect to the origin). If $\rho_\Omega$ is positive and continuous, then $\Omega$ is called a star body with respect to the origin. We write $\mathcal{S}^n$ for the space of all star bodies in $\mathbb{R}^n$. For real $a,b\geq 0$, the radial combination, $a\Omega_1\widetilde{+}b\Omega_2\in\mathbb{R}^n$, is the compact star-shaped set defined by
\begin{align*}
a\cdot \Omega_1\widetilde{+}b\cdot \Omega_2=\bigg\{ay_1+by_2:y_1\in\Omega_1~~\text{and}~~y_2\in\Omega_2, \text{whenever}~~y_1\cdot y_2=|y_1||y_2| \bigg\}.
\end{align*}
Obviously, $y_1\cdot y_2=|y_1||y_2|$ means that either $y_2=ky_1$ or $y_1=ky_2$ for some $k\geq 0$. The radial function of radial combination of two star-shaped sets is the combination of their radial functions, i.e.,
\begin{align*}
\rho(a\Omega_1\widetilde{+}b\Omega_2,\cdot)=a\rho(\Omega_1,\cdot)+b\rho(\Omega_2,\cdot).
\end{align*}
For real $p$, the radial $p$-combination $a\cdot \Omega_1\widetilde{+}_pb\cdot\Omega_2$ is defined by Lutwak \cite{LE}
\begin{align*}
\rho(a\Omega_1\widetilde{+}_pb\Omega_2,\cdot)^p=a\rho(\Omega_1,\cdot)^p+b\rho(\Omega_2,\cdot)^p, \quad p\neq 0;
\end{align*}
\begin{align*}
\rho(a\Omega_1\widetilde{+}_0b\Omega_2,\cdot)=\rho(\Omega_1,\cdot)^a\rho(\Omega_2,\cdot)^b, \quad p= 0.
\end{align*}
\subsection{Gauss map and radial gauss map of a convex body}
For a convex body $\Omega\in\mathbb{R}^n$, its support hyperplane with outward unit normal vector $\xi\in S^{n-1}$ is represented by
\begin{align*}
H(\Omega,\xi)=\{y\in\mathbb{R}^n:y\cdot \xi=h(\Omega,\xi)\}.
\end{align*}
A boundary point of $\Omega$ which only has one supporting hyperplane is called a regular point, otherwise, it is a singular point. The set of singular points is denoted as $\sigma \Omega$, it is
well known that $\sigma \Omega$ has spherical Lebesgue measure 0.

For $\sigma\subset\partial\Omega$, the spherical image of $\sigma$ is denoted by
\begin{align*}
\bm\nu_\Omega(\sigma)=\{v\in S^{n-1}:y\in H_\Omega(v)~~\text{for some}~~y\in\sigma\}\subset S^{n-1}.
\end{align*}
For a Borel set $\eta\subset S^{n-1}$, the reverse spherical image of $\eta$ is defined by
\begin{align*}
\bm x_\Omega(\eta)=\{y\in \partial\Omega:y\in H_\Omega(v)~~\text{for some}~~v\in\eta\}\subset \partial\Omega.
\end{align*}

For a Borel set $\eta\subset S^{n-1}$, the surface area measure of $\Omega$ is defined as
\begin{align*}S(\Omega,\eta)=\mathcal{H}^{n-1}(\nu^{-1}_{\Omega}(\eta)),\end{align*}
where $\mathcal{H}^{n-1}$ is the $(n-1)$-dimensional Hausdorff measure. The Gauss map $\nu_\Omega:y\in\partial \Omega\setminus \sigma \Omega\rightarrow S^{n-1}$ is represented by
\begin{align*}\nu_\Omega(y)=\{\xi\in S^{n-1}:y\cdot \xi=h_\Omega(\xi)\}.\end{align*}
Here, $\partial \Omega\setminus \sigma \Omega$ is abbreviated as $\partial^\prime\Omega$, something we will often do. If one views the reciprocal Gauss curvature of a smooth convex body as a function of the outer unit normals of the body, then the surface area measure is extension to an arbitrary convex body (that is not necessarily smooth) of the reciprocal Gauss curvature. In fact, if
$\partial\Omega$ is of class $C^2$ and has everywhere positive curvature, then the surface area measure has a positive density,
\begin{align}\label{eq201}
\frac{dS(\Omega,v)}{dv}=\det(h_{ij}(v)+h_\Omega(v)\delta_{ij}),
\end{align}
where $h_{ij}$ is the Hessian matrix of $h_\Omega$ with respect to an orthonormal frame on $S^{n-1}$, $\delta_{ij}$ is the Kronecker delta, the determinant is precisely the reciprocal Gauss curvature of
$\partial\Omega$ at the point of $\partial\Omega$ whose outer unit normal is $v$, where the Radon-Nikodym derivative is with respect to the spherical Lebesgue measure.

Correspondingly, for a Borel set $\eta\subset S^{n-1}$, its inverse Gauss map is denoted by $\nu_\Omega^{-1}$,
\begin{align*}\nu_\Omega^{-1}(\eta)=\{y\in\partial \Omega:\nu_\Omega(y)\in\eta\}.\end{align*}
Specially, for a convex hypersurface $\partial\Omega$ of class $C^2$, the support function of $\Omega$ can be stated as
\begin{align*}
h(\Omega,x)=x\cdot \nu_\Omega^{-1}(x)=\nu_\Omega(X(x))\cdot X(x), \quad X(x)\in \partial\Omega.
\end{align*}
Moreover, the gradient of $h(\Omega, \cdot)$ satisfies
\begin{align}\label{eq202}
\nabla h(\Omega,x)=\nu_\Omega^{-1}(x).
\end{align}

For $g\in C(S^{n-1})$,
\begin{align*}
\int_{\partial\Omega\setminus \sigma\Omega}g(\nu_\Omega(y))d\mathcal{H}^{n-1}(y)=\int_{S^{n-1}}g(v)dS(\Omega,v).
\end{align*}
Furthermore, for $\mathcal{H}^n$ almost all $X\in\partial\Omega$,
\begin{align*}
\nabla u(X)=-|\nabla u(X)|\nu_\Omega(X)\quad \text{and}\quad |\nabla u|\in L^q(\partial\Omega, \mathcal{H}^n).
\end{align*}

For $\Omega\in\mathcal{K}_o^n$, define the radial map of $\Omega$,
\begin{align*}
r_\Omega:S^{n-1}\rightarrow\partial\Omega, \quad \text{by} \quad r_\Omega(v)=\rho_\Omega(v)v\in\partial\Omega \quad \text{for}\quad v\in S^{n-1}.
\end{align*}
Note that $r_\Omega^{-1}:\partial\Omega\rightarrow S^{n-1}$ is just the restriction to $\partial\Omega$ of the map $y\mapsto \bar{y}=\frac{y}{|y|}$.

For $\omega\subset S^{n-1}$, define the radial Gauss image of $\omega$ by
\begin{align*}
\bm\alpha_\Omega(\omega)=\bm\nu_\Omega(r_\Omega(\omega))\subset S^{n-1}.
\end{align*}
Thus, for $v\in S^{n-1}$, one has 
\begin{align*}
\bm\alpha_\Omega(v)=\{\xi\in S^{n-1}:r_\Omega(v)\in H_\Omega(\xi)\}.
\end{align*}

Define the radial Gauss map of a convex body $\Omega\in\mathcal{K}_o^n$,
\begin{align*}
\alpha:S^{n-1}\setminus\omega_\Omega\rightarrow S^{n-1}~~\text{by}~\alpha_\Omega=\nu_\Omega\circ r_\Omega,
\end{align*}
where $\omega_\Omega=r^{-1}_\Omega(\sigma_\Omega)$. Since $r^{-1}_\Omega$ is a bi-Lipschitz map between the spaces $\partial\Omega$ and $S^{n-1}$, it follows that $\omega_\Omega$ has spherical Lebesgue measure $0$. Observe that if $v\in S^{n-1}\setminus\omega_\Omega$, then $\bm\alpha_\Omega(v)$ contains only the element $\alpha_\Omega(v)$. Note that since both $\nu_\Omega$ and $r_\Omega$ are continuous, $\alpha_\Omega$ is continuous.

For a Borel set $\omega\subset S^{n-1}$, $\bm\alpha_\Omega(\omega)$ denotes its radail Gauss image and is defined as follows:
\begin{align*}
\bm\alpha_\Omega(\omega)=\bigg\{v\in S^{n-1}:\rho_\Omega(\xi)(\xi\cdot v)=h_\Omega(v)\bigg\},
\end{align*}
for $v\in\omega$. If a Borel set $\omega$ has only one element $v$, then we will abbreviate $\bm {\alpha}_\Omega(\{v\})$ as $\bm{\alpha}_\Omega(v)$. Denoted by $\omega_\Omega$ the subset of $S^{n-1}$ which makes $\bm{\alpha}_\Omega(v)$ contain more than one element for each $v\in\omega$. It is well known that $\omega_\Omega$ has the spherical Lebesgue measure $0$. Note that if $v\in S^{n-1}\setminus\omega_\Omega$, then  $\bm{\alpha}_\Omega(v)$ contains only the element $\alpha_\Omega(v)$.

For $\mathcal{H}^{n-1}$-integrable function $g:\partial\Omega\rightarrow\mathbb{R}$,
\begin{align*}
\int_{\partial\Omega}g(y)d\mathcal{H}^{n-1}(y)=\int_{S^{n-1}}g(\rho_\Omega(v)v)J(v)dv,
\end{align*}
where $J$ \cite{HY} is defined $\mathcal{H}^{n-1}$-a.e. on $S^{n-1}$ by
\begin{align*}
J(v)=\frac{\rho_\Omega(v)^n}{h_\Omega(\alpha_\Omega(v))}.
\end{align*}

For a Borel set $\eta\subset S^{n-1}$, its reverse radial Gauss image $\bm\alpha_\Omega^*(\eta)$ is represented as follows:
\begin{align*}
\bm\alpha_\Omega^*(\eta)=\bigg\{v\in S^{n-1}:\rho_\Omega(v)(\xi\cdot v)=h_\Omega(v)\bigg\},
\end{align*}
for $v\in\eta$. When a Borel set $\eta$ has only one element $v$, we will abbreviate $\bm\alpha_\Omega^*(\{v\})=\bm\alpha_\Omega^*(v)$. Denoted by $\eta_\Omega$ the subset of $S^{n-1}$ which makes $\bm\alpha_\Omega^*(v)$ contain more than one element for each $v\in\eta$. It is well known that the spherical Lebesgue measure of $\eta_\Omega$ is $0$.

For $v\in S^{n-1}$ and $\eta\subset S^{n-1}$, we have
\begin{align*}
v\in \bm\alpha_\Omega^*(\eta)~~\text {if and only if}~~\bm\alpha_\Omega(v)\cap \eta\neq \emptyset.
\end{align*}
Thus, $\bm\alpha_\Omega^*$ is monotone non-decreasing with respect to the set inclusion.

If we write $\bm\alpha_\Omega^*(\{v\})$ by $\bm\alpha_\Omega^*(v)$, this yields that
\begin{align*}
\omega\in\bm\alpha_\Omega^*(v)~~\text {if and only if}~~v\in\bm\alpha_\Omega(\omega).
\end{align*}
\subsection{Wull shapes and convex hulls}
Denote by $C(S^{n-1})$ the set of continuous functions on $S^{n-1}$ which is often equipped with the metric induced by the maximal norm. We write $C^+(S^{n-1})$ for the set of strictly positive functions in $C(S^{n-1})$. For any nonnegative $f\in C(S^{n-1})$, the Aleksandrov body is defined by
\begin{align*}
[f]=\bigcap_{v\in S^{n-1}}\bigg\{y\in \mathbb{R}^n:y\cdot v\leq f(v)\bigg\},
\end{align*}
the set is Wulff shape associated with $f$. Obviously, $[f]$ is a compact convex set containing the origin. If $\Omega$ is a compact convex set containing the origin, then $\Omega=[h_{\Omega}]$. The Aleksandrov convergence lemma is shown as follows: if the sequence $f_i\in C^+(S^{n-1})$ converges uniformly to $f\in C^+(S^{n-1})$, then $\lim_{i\rightarrow\infty}[f_i]=[f]$. The convex hull $\langle \rho \rangle$ generated by $\rho$ is a convex body defined by, for $\rho\in C^+(S^{n-1})$,
\begin{align*}
\langle \rho \rangle=\text{conv}\bigg\{\rho(v)v,v\in S^{n-1}\bigg\}.
\end{align*}

Clearly, $[f]^*=\langle \frac{1}{f} \rangle$ and if $\Omega\in\mathcal{K}_o^n$, $\langle \rho_\Omega \rangle=\Omega$.

Let $\Theta\subset S^{n-1}$ be a closed set, $f:\Theta\rightarrow\mathbb{R}$ be continuous, $\delta>0$ and $h_s:\Theta\rightarrow (0,\infty)$ be a continuous function is defined for any $s\in(-\delta,\delta)$ by (see \cite{HY}),
\begin{align*}
\log h_s(v)=\log h(v)+sf(v)+o(s,v),
\end{align*}
for any $v\in \Theta$ and the function $o(s,\cdot):\Theta\rightarrow \mathbb{R}$ is continuous and $\lim_{s\rightarrow 0}o(s,\cdot)/s=0$ uniformly on $\Theta$. Denoted by $[h_s]$ the Wulff shape determined by $h_s$, we shall call $[h_s]$ a logarithmic family of the Wulff shapes formed by $(h,f)$. On occasion, we shall write $[h_s]$ as $[h,f,s]$, and if $h$ happens to be the support function of a convex body $\Omega$ perhaps as $[\Omega,f,s]$, or as  $[\Omega,f,o,s]$, if required for clarity.

Let $g:\Theta\rightarrow \mathbb{R}$ be continuous and $\delta >0$. Let $\rho_s:\Theta\rightarrow (0,\infty)$ be a continuous function defined for each $s\in(-\delta,\delta)$ and each $v\in \Theta$ by
\begin{align*}
\log \rho_s(v)=\log \rho(v)+sg(v)+o(s,v).
\end{align*}
Denoted by $\langle \rho_s\rangle$ the convex hull generated by $\rho_s$, we shall call $\langle \rho_s\rangle$ a logarithmic family of the convex hulls generated by $(\rho,g)$. On occasion $\langle \rho_s\rangle$ as $\langle \rho, g, s\rangle$, and if $\rho$ happens to be the radial function of a convex body $\Omega$ perhaps as $\langle\Omega,g,s\rangle$, or as  $\langle\Omega,g,o,s\rangle$, if required for clarity.

The following lemma \cite[Lemma 4.2]{HY} shows that the support functions of a logarithmic family of convex hulls are differentiable with respect to the variational variable.
\begin{lemma}\label{lem21}
Let $\Theta\subset S^{n-1}$ be a closed set that is not contained in any closed hemisphere of $S^{n-1}$, $\rho_0:\Theta\rightarrow(0,\infty)$ and $g:\Theta\rightarrow\mathbb{R}$ be continuous. If $\langle \rho_s \rangle$ is a logarithmic family of convex hulls of $(\rho_0,g)$, then, for $p\in\mathbb{R}$,
\begin{align*}
\lim_{s\rightarrow 0}\frac{ h^{-p}_{\langle \rho_s \rangle}(v)- h^{-p}_{\langle \rho_0 \rangle}(v)}{s}=-ph^{-p}_{\langle \rho_0 \rangle}(v)g(\alpha^*_{\langle \rho_0 \rangle}(v)),
\end{align*}
for all $v\in S^{n-1}\setminus \eta_{\langle \rho_0 \rangle}$. Moreover, there exist $\delta_0>0$ and $M>0$ so that
\begin{align*}
|h^{-p}_{\langle \rho_s \rangle}(v)-h^{-p}_{\langle \rho_0 \rangle}(v)|\leq M|s|,
\end{align*}
for all $v\in S^{n-1}$ and all $s\in(-\delta_0,\delta_0)$.
\end{lemma}
\subsection{Gauss curvature on convex hypersurface } Suppose that $\Omega$ is parameterized by the inverse Gauss map $X:S^{n-1}\rightarrow \Omega$, that is $X(x)=\nu_\Omega^{-1}(x)$. Then, the support function $h$ of $\Omega$ can be computed by
\begin{align}\label{eq203}h(x)=x\cdot X(x) , \ \ x\in S^{n-1},\end{align}
where $x$ is the outer normal of $\Omega$ at $X(x)$. Let $\{e_1, e_2, \cdots, e_{n-1}\}$ be an orthonormal frame on $S^{n-1}$, denote $e_{ij}$ by the standard metric on the sphere $S^{n-1}$. Differentiating (\ref{eq203}), there has
\begin{align*}\nabla_ih=\nabla_ix\cdot X(x)+ x\cdot \nabla_iX(x),\end{align*}
since $\nabla_iX(x)$ is tangent to $\Omega$ at $X(x)$, thus,
\begin{align*}\nabla_ih=\nabla_ix\cdot X(x).\end{align*}

By differentiating (\ref{eq203}) twice, the second fundamental form $A_{ij}$ of $\Omega$ can be computed in terms of the support function,
\begin{align}\label{eq204}A_{ij} = \nabla_{ij}h + he_{ij},\end{align}
where $\nabla_{ij}=\nabla_i\nabla_j$ denotes the second order covariant derivative with respect to $e_{ij}$. The induced metric matrix $g_{ij}$ of $\Omega$ can be derived by Weingarten's formula,
\begin{align}\label{eq205}e_{ij}=\nabla_ix\cdot \nabla_jx= A_{ik}A_{lj}g^{kl}.\end{align}
The principal radii of curvature are the eigenvalues of the matrix $b_{ij} = A^{ik}g_{jk}$. Considering a smooth local orthonormal frame on $S^{n-1}$, by virtues of (\ref{eq204}) and (\ref{eq205}), there is
\begin{align}\label{eq206}b_{ij} = A_{ij} = \nabla_{ij}h + h\delta_{ij}.\end{align}
Then, the Gauss curvature $\mathcal{K}(x)$ of $X(x)\in\partial\Omega$ is given by
\begin{align}\label{eq207}\mathcal{K}(x) = (\det (\nabla_{ij}h + h\delta_{ij} ))^{-1}.\end{align}

\section{\bf The $p$-th dual $q$-torsional measure and variational formula}\label{sec3}
\subsection{\bf The $p$-th dual $q$-torsional measure}
Firstly, we state the following variational formula for $q$-torsional rigidity was proved in \cite{HJ2}.
\begin{lemma}\label{lem31}\cite[Lemma 3.2]{HJ2} Let $\mathcal{O}\subset \mathbb{R}$ be an interval containing the origin in its interior, and let $h_s(\xi)=h(s,\xi):\mathcal{O}\times S^{n-1}\rightarrow (0,\infty)$ and $h_s(\xi)\in\mathcal{E}$ such that the convergence in
\begin{align*}
h^\prime(0,\xi)=\lim_{s\rightarrow 0^+}\frac{h(s,\xi)-h(0,\xi)}{s}
\end{align*}
is uniform on $S^{n-1}$. Then
\begin{align*}
\frac{d\widetilde{T}_q([h_s])}{ds}\bigg|_{s=0^+}=\int_{S^{n-1}}h^\prime_+(0,\xi)d\mu_q^{tor}([h_0],\xi).
\end{align*}
Obviously, if $f\in C(S^{n-1})$, then
\begin{align}\label{eq301}
\int_{S^{n-1}}f(v)d\mu^{tor}_q(\Omega,v)=\int_{S^{n-1}}f(\alpha_\Omega(v))H(v)^qdv,
\end{align}
thus, from (\ref{eq105}), we obtain
\begin{align*}
\widetilde{T}_q(\Omega)=\frac{b}{a}\int_{S^{n-1}}\rho^n_{\Omega}(v)|\nabla u|^qdv,
\end{align*}
where $\mathcal{E}=\{h\in C_+^{2,\alpha}(S^{n-1}):h_{ij}+h\delta_{ij}~~\text{is positive definite}\}$ and $H(v)=|\nabla u(r_\Omega(v))|J(v)^{\frac{1}{q}}$, $r_\Omega(v)=\rho_\Omega(v)v$ and $v\in S^{n-1}$.
\end{lemma}

\begin{proposition}\label{pro32}Let $\rho_0:S^{n-1}\rightarrow \mathbb{R}$ and $g:S^{n-1}\rightarrow \mathbb{R}$ be continuous. If $\langle \rho_s\rangle$ is a logarithmic family of convex hulls of $(\rho_0,g)$, then, for $q>1$,
\begin{align*}
\lim_{s\rightarrow 0}\frac{\widetilde{T}_q(\langle\rho_s\rangle)-\widetilde{T}_q(\langle\rho_0\rangle)}{s}=\int_{S^{n-1}}g(v)\rho^n_{\langle\rho_0\rangle}(v)|\nabla u(r_{\langle\rho_0\rangle}(v))|^qdv.
\end{align*}
\end{proposition}
\begin{proof}
Using the dominated convergence theorem,  Lemma \ref{lem31}, Lemma \ref{lem21} and (\ref{eq301}), we get
\begin{align*}
&\lim_{s\rightarrow 0}\frac{\widetilde{T}_q(\langle\rho_s\rangle)-\widetilde{T}_q(\langle\rho_0\rangle)}{s}\\
=&\int_{S^{n-1}}\lim_{s\rightarrow 0}\frac{h_{\langle \rho_s\rangle}(\xi)-h_{\langle \rho_0\rangle}(\xi)}{s}d\mu_q^{tor}(\langle \rho_0\rangle,\xi)\\
=&\int_{S^{n-1}}g(\alpha^*_{\langle \rho_0\rangle}(\xi))h_{\langle \rho_0\rangle}(\xi)d\mu_q^{tor}(\langle \rho_0\rangle,\xi)\\
=&\int_{S^{n-1}}g(v)h_{\langle \rho_0\rangle}(\alpha_{\langle \rho_0\rangle}(v))H(v)^qdv\\
=&\int_{S^{n-1}}g(v)\rho_{\langle \rho_0\rangle}^n(v)|\nabla u(r_{\langle\rho_0\rangle}(v))|^qdv.
\end{align*}
\end{proof}
\begin{corollary}\label{cor33}Let $\Omega_1, \Omega_2\in\mathcal{K}_o^n$, $q>1$ and $p\neq 0$, then
\begin{align*}
\lim_{s\rightarrow 0}\frac{\widetilde{T}_q(\langle\rho_{\Omega_1\widetilde{+}_ps\cdot\Omega_2}\rangle)-\widetilde{T}_q(\langle\rho_{\Omega_1}\rangle)}{s}=
\frac{1}{p}\int_{S^{n-1}}\rho_{\Omega_2}(v)^p\rho_{\Omega_1}(v)^{n-p}|\nabla u(r_{\Omega_1}(v))|^qdv.
\end{align*}
\end{corollary}
\begin{proof}
From
\begin{align}\label{eq302}
\rho_{\Omega_1\widetilde{+}_ps\cdot\Omega_2}=(\rho^p_{\Omega_1}+s\rho^p_{\Omega_2})^{\frac{1}{p}},
\end{align}
we obtain
\begin{align}\label{eq303}
\nonumber\log(\rho_{\Omega_1\widetilde{+}_ps\cdot\Omega_2})=&\frac{1}{p}[\log(\rho^p_{\Omega_1}+s\rho^p_{\Omega_2})]\\
\nonumber=&\frac{1}{p}\bigg[\log \bigg(\rho^p_{\Omega_1}\bigg(1+\frac{s\rho^p_{\Omega_2}}{\rho^p_{\Omega_1}}\bigg)\bigg)\bigg]\\
\nonumber=&\log\rho_{\Omega_1}+\frac{1}{p}\log\bigg(1+\frac{s\rho^p_{\Omega_2}}{\rho^p_{\Omega_1}}\bigg)\\
=&\log\rho_{\Omega_1}+s\frac{\rho^p_{\Omega_2}}{p\rho^p_{\Omega_1}}+o(s,\cdot).
\end{align}
Since $\Omega_1, \Omega_2\in\mathcal{K}_o^n$, the logarithmic family of convex hulls $\langle\rho_{\Omega_1\widetilde{+}_ps\cdot\Omega_2}\rangle=\log\rho_{\Omega_1}+s\frac{\rho^p_{\Omega_2}}{p\rho^p_{\Omega_1}}+o(s,\cdot)$. Let $\langle\rho_0\rangle=\Omega_1$ and $g=\frac{\rho^p_{\Omega_2}}{p\rho^p_{\Omega_1}}$, thus, the desired result now follows directly from Proposition \ref{pro32} and formula (\ref{eq303}).
\end{proof}

Using the above variational formula for $q$-torsional rigidity, we can define the dual mixed $q$-torsional rigidity as follows: Let $q>1$, $p\in\mathbb{R}$ and convex bodies $\Omega_1, \Omega_2\in\mathcal{K}_o^n$, the dual mixed $q$-torsional rigidity $\widetilde{Q}_{q,p}(\Omega_1,\Omega_2)$ is defined by
\begin{align}\label{eq304}
\widetilde{Q}_{q,p}(\Omega_1,\Omega_2)=\frac{b}{a}\int_{S^{n-1}}\rho_{\Omega_2}(v)^p\rho_{\Omega_1}(v)^{n-p}|\nabla u(r_{\Omega_1}(v))|^qdv.
\end{align}

When $\Omega_1=\Omega_2$, the dual $q$-torsional rigidity of $\Omega_1$ will be shown to be the special case as follows:
\begin{align*}
\widetilde{Q}_{q}(\Omega_1)=\widetilde{Q}_{q,p}(\Omega_1,\Omega_1)=\frac{b}{a}\int_{S^{n-1}}\rho_{\Omega_1}(v)^n|\nabla u(r_{\Omega_1}(v))|^qdv=\widetilde{T}_q(\Omega_1).
\end{align*}

Let $\Omega_2=\mathcal{B}$ ($\mathcal{B}$ is unit ball with $\rho_{\mathcal{B}}(v)=1$) in (\ref{eq304}), the $p$-th dual $q$-torsional rigidity of $\Omega_1$ is defined by
\begin{align}\label{eq305}
\widetilde{Q}_{q,n-p}(\Omega_1)=\frac{b}{a}\int_{S^{n-1}}\rho_{\Omega_1}(v)^p|\nabla u(r_{\Omega_1}(v))|^qdv.
\end{align}

The definition of $p$-th dual $q$-torsional measure has be needed in the introduction. For convenience, the definition of $p$-th dual $q$-torsional measure will be restated as follows.
\begin{definition}\label{def34}
Let $p\in\mathbb{R}$, $q>1$ and $\Omega\in\mathcal{K}_o^n$, we define the $p$-th dual $q$-torsional measure by
\begin{align*}
\widetilde{Q}_{q,n-p}(\Omega,\eta)=&\frac{b}{a}\int_{\alpha_\Omega^*(\eta)}\rho^p_{\Omega}(v)|\nabla u(r_\Omega(v))|^qdv\\
=&\frac{b}{a}\int_{S^{n-1}}\bm1_{\alpha_\Omega^*(\eta)}\rho^p_{\Omega}(v)|\nabla u(r_\Omega(v))|^qdv,
\end{align*}
for each Borel set $\eta\subset S^{n-1}$ and $r_\Omega(v)=\rho_\Omega(v)v$.
\end{definition}

Note that, if $p=n$, then
\begin{align*}
\widetilde{Q}_{q}(\Omega,\cdot)=\frac{b}{a}h_{\Omega}\mu_q^{tor}(\Omega,\cdot),
\end{align*}
where $\mu_q^{tor}(\Omega,\cdot)$ denotes the $q$-torsional measure of a convex body $\Omega$, see (\ref{eq106}).
\subsection{\bf The $p$-th dual $q$-torsional measures for special classes of convex bodies}
\begin{lemma}\label{lem35}
Let $\Omega\in\mathcal{K}_o^n$, $q>1$ and $p\in\mathbb{R}$. For each function $g:S^{n-1}\rightarrow\mathbb{R}$, $\eta\subset S^{n-1}$, then
\begin{align}\label{eq306}
\int_{S^{n-1}}g(\xi)d\widetilde{Q}_{q,n-p}(\Omega,\xi)=\frac{b}{a}\int_{S^{n-1}}g(\alpha_\Omega(v))\rho_{\Omega}(v)^p|\nabla u(r_\Omega(v))|^qdv,
\end{align}
\begin{align}\label{eq307}
\int_{S^{n-1}}g(\xi)d\widetilde{Q}_{q,n-p}(\Omega,\xi)=\frac{b}{a}\int_{\partial^\prime\Omega}g(\nu_\Omega(y))y\cdot\nu_\Omega(y)|y|^{p-n}|\nabla u(y)|^qd\mathcal{H}^{n-1}(y),
\end{align}
and
\begin{align}\label{eq308}
\widetilde{Q}_{q,n-p}(\Omega,\eta)=\frac{b}{a}\int_{y\in\nu_\Omega^{-1}(\eta)}y\cdot\nu_\Omega(y)|y|^{p-n}|\nabla u(y)|^qd\mathcal{H}^{n-1}(y).
\end{align}
\end{lemma}
\begin{proof}
Firstly, let's provide a proof for (\ref{eq306}), this proof method refers to \cite [Lemma 3.3]{HY}. Assuming $\psi$ is a simple function on $S^{n-1}$ given by
\begin{align*}
\psi=\sum_{i=1}^mc_i\bm1_{\eta_i}
\end{align*}
with $c_i\in\mathbb{R}$ and Borel set $\eta_i\subset S^{n-1}$. By using Definition \ref{def34} and \cite[Equation (2.21)]{HY}, we get
\begin{align*}
\int_{S^{n-1}}\psi(\xi)d\widetilde{Q}_{q,n-p}(\Omega,\xi)=&\int_{S^{n-1}}\sum_{i=1}^mc_i\bm1_{\eta_i}(\xi)d\widetilde{Q}_{q,n-p}(\Omega,\xi)\\
=&\sum_{i=1}^mc_i\widetilde{Q}_{q,n-p}(\Omega,\eta_i)\\
=&\frac{b}{a}\int_{S^{n-1}}\sum_{i=1}^mc_i\bm1_{\bm\alpha^*_\Omega(\eta_i)}(v)\rho_{\Omega}(v)^p|\nabla u(r_\Omega(v))|^qdv\\
=&\frac{b}{a}\int_{S^{n-1}}\sum_{i=1}^mc_i\bm1_{\eta_i}(\alpha_\Omega(v))\rho_{\Omega}(v)^p|\nabla u(r_\Omega(v))|^qdv\\
=&\frac{b}{a}\int_{S^{n-1}}\sum_{i=1}^m\psi(\alpha_\Omega(v))\rho_{\Omega}(v)^p|\nabla u(r_\Omega(v))|^qdv.
\end{align*}

Note that we have established (\ref{eq306}) for simple functions, for a bounded Borel $g$, we choose a sequence of simple functions $\psi_k$ that converge to $g$, uniformly. Then $\psi_k\circ \alpha_\Omega$ to $g\circ \alpha_\Omega$ a.e. with respect to the spherical Lebesgue measure. Since $g$ is a Borel function on $S^{n-1}$ and the radial Gauss map $\alpha_\Omega$ is continuous on $S^{n-1}\setminus\eta_\Omega$, the composite function $g\circ \alpha_\Omega$ is a Borel function on $S^{n-1}\setminus\eta_\Omega$. Hence, $g$ and $g\circ \alpha_\Omega$ are Lebesgue integrable on $S^{n-1}$ because $g$
is bounded and $\eta_\Omega$ has the Lebesgue measure zero. Taking the limit $k\rightarrow\infty$ establishes (\ref{eq306}).

Next, we give a proof of (\ref{eq307}). Let $f=g\circ\alpha_\Omega$, then, as shown in the proof of (\ref{eq306}), $f$ is bounded and Lebesgue integrable on $S^{n-1}$. Thus, the desired (\ref{eq307}) follows immediately from (\ref{eq306}) and \cite[Equation (2.31)]{HY}. In fact, let $y\in\partial\Omega$ and $v=\overline{y}=\frac{y}{|y|}$,
\begin{align*}
\int_{S^{n-1}}g(\xi)d\widetilde{Q}_{q,n-p}(\Omega,\xi)=&\frac{b}{a}\int_{S^{n-1}}g(\alpha_\Omega(v))\rho_{\Omega}(v)^p|\nabla u(r_\Omega(v))|^qdv\\
=&\frac{b}{a}\int_{\partial^\prime\Omega}g(\alpha_\Omega(\overline{y}))|y|^{p-n}(y\cdot\nu_\Omega(y))|\nabla u(y)|^qd\mathcal{H}^{n-1}(y)\\
=&\frac{b}{a}\int_{\partial^\prime\Omega}g(\nu_\Omega(y))|y|^{p-n}(y\cdot\nu_\Omega(y))|\nabla u(y)|^qd\mathcal{H}^{n-1}(y).
\end{align*}
This gives (\ref{eq307}).

Finally, the proof of (\ref{eq308}) is established. Let $g=\bm1_\eta$ in (\ref{eq307}) and the fact that $\nu_\Omega(y)\in\eta\Leftrightarrow y\in\nu_\Omega^{-1}(\eta)$, for almost all $y$ with respect to the spherical Lebesgue measure. Consequently, we directly attain (\ref{eq308}).
\end{proof}

We conclude with three observations regarding the $p$-th dual $q$-torsional measures.

(i) Let $P\in\mathcal{K}_o^n$ be a polytope with outer unit normals $v_1,\cdots,v_m$, $\triangle_i$ be the cone that consists of all of the rays emanating from the origin and passing through the facet of $P$ whose outer unit normal is $v_i$. Then, recalling that we abbreviate $\bm\alpha_P^*(\{v_i\})$ by $\bm\alpha_P^*(v_i)$, we have
\begin{align}\label{eq309}
\bm\alpha_P^*(v_i)=S^{n-1}\cap \triangle_i.
\end{align}
If $\eta\subset S^{n-1}$ is a Borel set such that $\{v_1,\cdots,v_m\}\cap\eta=\emptyset$, then $\bm\alpha_P^*(\eta)$ has the spherical Lebesgue measure zero. Thus, the $p$-th dual $q$-torsional measure $\widetilde{Q}_{q,n-p}(P,\cdot)$ is discrete and concentrated on $\{v_1,\cdots,v_m\}$. By Definition \ref{def34} and equality (\ref{eq309}), we have
\begin{align*}
\widetilde{Q}_{q,n-p}(P,\cdot)=\sum_{i=1}^mc_i\delta_{v_i},
\end{align*}
where $\delta_{v_i}$ defines the delta measure concentrated at the point $v_i$ on $S^{n-1}$, and
\begin{align*}
c_i=\frac{b}{a}\int_{S^{n-1}\cap \triangle_i}\rho_P(v)^p|\nabla u(r_P(v))|^qdv.
\end{align*}

(ii) Assume that $\Omega\in\mathcal{K}_o^n$ is strictly convex. If $g:S^{n-1}\rightarrow\mathbb{R}$ is continuous, then (\ref{eq307}) and \cite[Equation (2.33)]{HY} yield
\begin{align*}
\int_{S^{n-1}}g(v)d\widetilde{Q}_{q,n-p}(\Omega,v)=&\frac{b}{a}\int_{\partial\Omega}y\cdot\nu_\Omega(y)g(\nu_\Omega(y))|y|^{p-n}|\nabla u(y)|^qd\mathcal{H}^{n-1}(y)\\
=&\frac{b}{a}\int_{S^{n-1}}g(v)|\overline{\nabla} h_{\Omega}(v)|^{p-n}h_{\Omega}(v)|\nabla u(\nu^{-1}_\Omega(v))|^qdS(\Omega,v).
\end{align*}
This shows that
\begin{align}\label{eq310}
d\widetilde{Q}_{q,n-p}(\Omega,\cdot)=\frac{b}{a}|\overline{\nabla} h_{\Omega}|^{p-n}h_{\Omega}|\nabla u|^qdS(\Omega,\cdot).
\end{align}

(iii) Assume that $\Omega\in\mathcal{K}^n_o$ has a $C^2$ boundary with everywhere positive curvature. Since in this case $S(\Omega,\cdot)$ is absolutely continuous with respect to the spherical Lebesgue measure, it follows that $\widetilde{Q}_{q,n-p}(\Omega,\cdot)$ is absolutely continuous with respect to the spherical Lebesgue measure, and from (\ref{eq310}) and (\ref{eq201}), we have
\begin{align}\label{eq311}
\frac{d\widetilde{Q}_{q,n-p}(\Omega,v)}{dv}=\frac{b}{a}|\overline{\nabla} h_{\Omega}(v)|^{p-n}h_{\Omega}(v)|\nabla u(\nu^{-1}_\Omega(v))|^q\det(h_{ij}(v)+h_\Omega(v)\delta_{ij}),
\end{align}
where $h_{ij}$ denotes the Hessian matrix of $h_\Omega$ with respect to an orthonormal frame on $S^{n-1}$, $\overline{\nabla}$ is the gradient in $\mathbb{R}^n$ with respect to the Euclidean metric and $\nabla$ is the gradient operator on $S^{n-1}$ with respect to the induced metric. Then for a function $h:\mathbb{R}^n\rightarrow\mathbb{R}$ which is differentiable at $v\in\mathbb{R}^n$, with $|v|=1$, we have
\begin{align*}
\overline{\nabla} h(v)=\nabla h(v)+h(v)v.
\end{align*}

\subsection{\bf Properties of the $p$-th dual $q$-torsional measure} In this subsection, we get some properties of the $p$-th dual $q$-torsional measure.
\begin{lemma}\label{lem36}
Let $\Omega\in\mathcal{K}_o^n$ and $p\in\mathbb{R}$, then the $p$-th dual $q$-torsional measure $\widetilde{Q}_{q,n-p}(\Omega,\cdot)$ is a Borel measure on $S^{n-1}$.
\end{lemma}
\begin{proof}
It is clear that $\widetilde{Q}_{q,n-p}(\Omega,\emptyset)=0$. We only need to prove the countable additivity. Namely, given a sequence of disjoint sets $\eta_i\subset S^{n-1}$, $i=1,2,\cdots$, with $\eta_i\cap \eta_j=\emptyset$ for $i\neq j$, the following formula holds:
\begin{align*}
\widetilde{Q}_{q,n-p}(\Omega,\cup_{i=1}^\infty\eta_i)=\sum_{i=1}^\infty\widetilde{Q}_{q,n-p}(\Omega,\eta_i).
\end{align*}
To this end, it follows from Definition \ref{def34} that for each Borel set $\eta_i\subset S^{n-1}$, one has
\begin{align*}
\widetilde{Q}_{q,n-p}(\Omega,\eta_i)=\frac{b}{a}\int_{\alpha_\Omega^*(\eta_i)}\rho^p_{\Omega}(v)|\nabla u(r_\Omega(v))|^qdv.
\end{align*}
By \cite[Lemmas 2.1-2.4]{HY}, the additivity for Lebesgue integral and fact that the spherical measure of $\omega_\Omega$ is zero, one has
\begin{align*}
\widetilde{Q}_{q,n-p}(\Omega,\cup_{i=1}^\infty\eta_i)=&\frac{b}{a}\int_{\alpha_\Omega^*(\cup_{i=1}^\infty\eta_i)}\rho^p_{\Omega}(v)|\nabla u(r_\Omega(v))|^qdv\\
=&\frac{b}{a}\int_{\cup_{i=1}^\infty\alpha_\Omega^*(\eta_i)}\rho^p_{\Omega}(v)|\nabla u(r_\Omega(v))|^qdv\\
=&\frac{b}{a}\int_{\cup_{i=1}^\infty\alpha_\Omega^*(\eta_i\setminus \omega_\Omega)}\rho^p_{\Omega}(v)|\nabla u(r_\Omega(v))|^qdv\\
=&\frac{b}{a}\sum_{i=1}^\infty\int_{\alpha_\Omega^*(\eta_i\setminus \omega_\Omega)}\rho^p_{\Omega}(v)|\nabla u(r_\Omega(v))|^qdv\\
=&\frac{b}{a}\sum_{i=1}^\infty\int_{\alpha_\Omega^*(\eta_i)}\rho^p_{\Omega}(v)|\nabla u(r_\Omega(v))|^qdv-\frac{b}{a}\sum_{i=1}^\infty\int_{\alpha_\Omega^*(\omega_\Omega)}\rho^p_{\Omega}(v)|\nabla u(r_\Omega(v))|^qdv\\
=&\frac{b}{a}\sum_{i=1}^\infty\int_{\alpha_\Omega^*(\eta_i)}\rho^p_{\Omega}(v)|\nabla u(r_\Omega(v))|^qdv\\
=&\sum_{i=1}^\infty\widetilde{Q}_{q,n-p}(\Omega,\eta_i).
\end{align*}
The countable additivity holds and hence $\widetilde{Q}_{q,n-p}(\Omega,\cdot)$ is a Borel measure.
\end{proof}
\begin{lemma}\label{lem37}
Let $\Omega\in\mathcal{K}_o^n$ and $p\in\mathbb{R}$, then the $p$-th dual $q$-torsional measure $\widetilde{Q}_{q,n-p}(\Omega,\cdot)$ is absolutely continuous with respect to the surface area measure $S(\Omega,\cdot)$.
\end{lemma}
\begin{proof}
Let $\eta\subset S^{n-1}$ be such that $S(\Omega,\eta)=0$, or equivalently, $\mathcal{H}^{n-1}(\nu_\Omega^{-1}(\eta))=0$. In this case, using (\ref{eq308}), we conclude that
\begin{align*}
\widetilde{Q}_{q,n-p}(\Omega,\eta)=\frac{b}{a}\int_{y\in\nu_\Omega^{-1}(\eta)}y\cdot\nu_\Omega(y)|y|^{p-n}|\nabla u(y)|^qd\mathcal{H}^{n-1}(y)=0,
\end{align*}
since we are integrating over a set of measure zero.
\end{proof}

\begin{lemma}\label{lem38}
Let $p\in\mathbb{R}$. If $\Omega_i\in\mathcal{K}_o^n$ with $\Omega_i\rightarrow\Omega_0\in\mathcal{K}_o^n$, then $\widetilde{Q}_{q,n-p}(\Omega_i,\cdot)\rightarrow\widetilde{Q}_{q,n-p}(\Omega_0,\cdot)$, weakly.
\end{lemma}
\begin{proof}
Let $g:S^{n-1}\rightarrow\mathbb{R}$ be continuous. From (\ref{eq306}), we know that
\begin{align*}
\int_{S^{n-1}}g(\xi)d\widetilde{Q}_{q,n-p}(\Omega_i,\xi)=\frac{b}{a}\int_{S^{n-1}}g(\alpha_{\Omega_i}(v))\rho_{\Omega_i}(v)^p|\nabla u(r_{\Omega_i}(v))|^qdv,
\end{align*}
for all $i$. The convergence $\Omega_i\rightarrow\Omega_0$ with respect to the Hausdorff metric implies that $\rho(\Omega_i,v)\rightarrow\rho(\Omega_0,v)$ uniformly on $S^{n-1}$. Since $\Omega_i, \Omega_0\in\mathcal{K}_o^n$, there are positive constants $c$ and $C$ such that for all $v\in S^{n-1}$ and all $i=1,2,\cdots$,
\begin{align*}
c\leq \rho(\Omega_i,v),\rho(\Omega_0,v)\leq C.
\end{align*}
For any given continuous function $g:S^{n-1}\rightarrow\mathbb{R}$ that there is a positive constant $M$ such that for any $i=1,2,\cdots$,
\begin{align*}
|g(\alpha_{\Omega_i})\rho^p(\Omega_i,\cdot)|\leq M \quad \text{and} \quad |g(\alpha_{\Omega_0})\rho^p(\Omega_0,\cdot)|\leq M.
\end{align*} 
From $\Omega_i\rightarrow \Omega_0$ and continuity of $r_\Omega$, we know that $r(\Omega_i,v)\rightarrow r(\Omega_0,v)$. The continuity of $\nabla u$ from Lemma \ref{lem53} on $\Omega_i,\Omega_0\in\mathcal{K}_o^n$ implies 
\begin{align*}|\nabla u(r_{\Omega_i}(v))|\leq C_1\quad\text{and}\quad|\nabla u(r_{\Omega_0}(v))|\leq C_1.
\end{align*}
Thus, the desired result directly from the Lemma 2.2 of \cite{HY} and dominated convergence theorem:
\begin{align*}
\frac{b}{a}\int_{S^{n-1}}g(\alpha_{\Omega_i}(v))\rho_{\Omega_i}(v)^p|\nabla u(r_{\Omega_i}(v))|^qdv\rightarrow\frac{b}{a}\int_{S^{n-1}}g(\alpha_{\Omega_0}(v))\rho_{\Omega_0}(v)^p|\nabla u(r_{\Omega_0}(v))|^qdv,
\end{align*}
from this it follows that $\widetilde{Q}_{q,n-p}(\Omega_i,\cdot)\rightarrow\widetilde{Q}_{q,n-p}(\Omega_0,\cdot)$, weakly.
\end{proof}
\subsection{Variational formulas for the $p$-th dual $q$-torsional rigidity}
\begin{theorem}\label{the39}
Let $\eta\subset S^{n-1}$ be a closed set not contained in any closed hemisphere of $S^{n-1}$, $\rho_0:\eta\rightarrow (0,\infty)$ and $g:\eta\rightarrow\mathbb{R}$ be continuous. If $\langle\rho_s\rangle$ is a logarithmic family of the convex hulls of $(\rho_0,g)$, then, for $p\neq 0$,
\begin{align*}
\lim_{s\rightarrow 0}\frac{\widetilde{Q}_{q,n-p}(\langle\rho_s\rangle^*)-\widetilde{Q}_{q,n-p}(\langle\rho_0\rangle^*)}{s}=-p(1+q)\int_\eta g(\xi)d\widetilde{Q}_{q,n-p}(\langle\rho_0\rangle^*,\xi).
\end{align*}
\end{theorem}
\begin{proof}This proof is similar to \cite[Theorem 4.4]{HY}, however, due to the existence of $|\nabla u|$, it is even more difficult than proof of \cite[Theorem 4.4]{HY}. Here, we omit \cite[page 364: lines 1-22]{HY} to only write the calculation parts. From (\ref{eq305}) and Lemma \ref{lem21}, we have
\begin{align*}
&\lim_{s\rightarrow 0}\frac{\widetilde{Q}_{q,n-p}(\langle\rho_s\rangle^*)-\widetilde{Q}_{q,n-p}(\langle\rho_0\rangle^*)}{s}=\frac{d}{ds}\widetilde{Q}_{q,n-p}(\langle\rho_s\rangle^*)\bigg|_{s=0}\\
=&\frac{b}{a}\int_{S^{n-1}}\bigg(\frac{d}{ds}\rho^p_{\langle\rho_s\rangle^*}(v)\bigg|_{s=0}|\nabla u(r_{\langle\rho_0\rangle^*}(v))|^q+\rho^p_{\langle\rho_0\rangle^*}(v)\frac{d}{ds}|\nabla u(r_{\langle\rho_s\rangle^*}(v))|^q\bigg|_{s=0}\bigg)dv\\
=&\frac{b}{a}\int_{S^{n-1}}\bigg(\frac{d}{ds}\rho^p_{\langle\rho_s\rangle^*}(v)\bigg|_{s=0}|\nabla u(r_{\langle\rho_0\rangle^*}(v))|^q+\rho^p_{\langle\rho_0\rangle^*}(v)\frac{d}{ds}|\nabla u(\rho_{\langle\rho_s\rangle^*}(v)v)|^q\bigg|_{s=0}\bigg)dv\\
=&\frac{b}{a}\int_{S^{n-1}}\bigg(\frac{d}{ds}h^{-p}_{\langle\rho_s\rangle}(v)\bigg|_{s=0}|\nabla u(r_{\langle\rho_0\rangle^*}(v))|^q+\rho^p_{\langle\rho_0\rangle^*}(v)\frac{d}{ds}|\nabla u(h^{-1}_{\langle\rho_s\rangle}(v)v)|^q\bigg|_{s=0}\bigg)dv\\
=&\frac{b}{a}\int_{S^{n-1}}\bigg(\lim_{s\rightarrow0}\frac{h^{-p}_{\langle\rho_s\rangle}(v)-h^{-p}_{\langle\rho_0\rangle}(v)}{s}|\nabla u(r_{\langle\rho_0\rangle^*}(v))|^q+\rho^p_{\langle\rho_0\rangle^*}(v)\frac{d}{ds}|\nabla u(h^{-1}_{\langle\rho_s\rangle}(v)v)|^q\bigg|_{s=0}\bigg)dv\\
=&\frac{b}{a}\int_{S^{n-1}\backslash \eta_0}-ph^{-p}_{\langle\rho_0\rangle}(v)g(\alpha^*_{\langle\rho_0\rangle}(v))|\nabla u(r_{\langle\rho_0\rangle^*}(v))|^qdv\\
&+\frac{b}{a}\int_{S^{n-1}}\rho^p_{\langle\rho_0\rangle^*}(v)\frac{d}{ds}|\nabla u(h^{-1}_{\langle\rho_s\rangle}(v)v)|^q\bigg|_{s=0}dv\\
=&\frac{b}{a}\int_{S^{n-1}\backslash \eta_0}-p\rho^{p}_{\langle\rho_0\rangle^*}(v)g(\alpha^*_{\langle\rho_0\rangle}(v))|\nabla u(r_{\langle\rho_0\rangle^*}(v))|^qdv\\
&+\frac{b}{a}\int_{S^{n-1}}\rho^p_{\langle\rho_0\rangle^*}(v)\frac{d}{ds}|\nabla u(h^{-1}_{\langle\rho_s\rangle}(v)v)|^q\bigg|_{s=0}dv.
\end{align*}
Recall that
\begin{align*}|\nabla u(h^{-1}_{\langle\rho_s\rangle}(v)v)|=-\nabla u(h^{-1}_{\langle\rho_s\rangle}(v)v)\cdot v.
\end{align*}
Thus,
\begin{align*}
&\frac{d}{ds}|\nabla u(h^{-1}_{\langle\rho_s\rangle}(v)v)|^q\bigg|_{s=0}\\
=&q|\nabla u(h^{-1}_{\langle\rho_0\rangle}(v)v)|^{q-1}\frac{d}{ds}|\nabla u(h^{-1}_{\langle\rho_s\rangle}(v)v)|\bigg|_{s=0}\\
=&-q|\nabla u(h^{-1}_{\langle\rho_0\rangle}(v)v)|^{q-1}\bigg((\nabla ^2u(h^{-1}_{\langle\rho_0\rangle}(v)v)\frac{d}{ds}(h^{-1}_{\langle\rho_s\rangle}(v)v))\cdot v+(\nabla \dot{u}(h^{-1}_{\langle\rho_0\rangle}(v)v))\cdot v\bigg)\\
=&-q|\nabla u(h^{-1}_{\langle\rho_0\rangle}(v)v)|^{q-1}\bigg((\nabla ^2u(h^{-1}_{\langle\rho_0\rangle}(v)v)[-h^{-1}_{\langle\rho_0\rangle}(v)g(\alpha^*_{\langle\rho_0\rangle}(v))]v)\cdot v+(\nabla \dot{u}(h^{-1}_{\langle\rho_0\rangle}(v)v))\cdot v\bigg)\\
=&-q|\nabla u(r_{\langle\rho_0\rangle^*}(v))|^{q-1}\bigg(\nabla ^2u(r_{\langle\rho_0\rangle^*}(v))[-\rho_{\langle\rho_0\rangle^*}(v)g(\alpha^*_{\langle\rho_0\rangle}(v))]+(\nabla \dot{u}(r_{\langle\rho_0\rangle^*}(v)))\cdot v\bigg)\\
=&q|\nabla u(r_{\langle\rho_0\rangle^*}(v))|^{q-1}\nabla ^2u(r_{\langle\rho_0\rangle^*}(v))\rho_{\langle\rho_0\rangle^*}(v)g(\alpha^*_{\langle\rho_0\rangle}(v))-q|\nabla u(r_{\langle\rho_0\rangle^*}(v))|^{q-1}(\nabla \dot{u}(r_{\langle\rho_0\rangle^*}(v)))\cdot v.\\
\end{align*}

Denote (see \cite{C0} or \cite{HY0})
\begin{align*} 
\frac{d}{ds}|\nabla u(r_{\langle\rho_s\rangle^*}(v))|^q\bigg|_{s=0}=&\frac{d}{ds}|\nabla u(\rho_{\langle\rho_s\rangle^*}(v)v)|^q\bigg|_{s=0}\\
=&\mathcal{D}(-pg(\alpha^*_{\langle\rho_0\rangle}(v))\rho^p_{\langle\rho_0\rangle^*}(v))\\
=&\mathcal{D}_1(-pg(\alpha^*_{\langle\rho_0\rangle}(v))\rho^p_{\langle\rho_0\rangle^*}(v))+\mathcal{D}_2(-pg(\alpha^*_{\langle\rho_0\rangle}(v))\rho^p_{\langle\rho_0\rangle^*}(v)) \end{align*}
with
\begin{align*}
\mathcal{D}_1(-pg(\alpha^*_{\langle\rho_0\rangle}(v))\rho^p_{\langle\rho_0\rangle^*}(v))=q|\nabla u(r_{\langle\rho_0\rangle^*}(v))|^{q-1}\nabla ^2u(r_{\langle\rho_0\rangle^*}(v))\rho_{\langle\rho_0\rangle^*}(v)g(\alpha^*_{\langle\rho_0\rangle^*}(v)),
\end{align*}
and
\begin{align*}
\mathcal{D}_2(-pg(\alpha^*_{\langle\rho_0\rangle}(v))\rho^p_{\langle\rho_0\rangle^*}(v))=-q|\nabla u(r_{\langle\rho_0\rangle^*}(v))|^{q-1}(\nabla \dot{u}(r_{\langle\rho_0\rangle^*}(v)))\cdot v.
\end{align*}
We can see that $\mathcal{D}$ is a self-adjoint operator on $S^{n-1}$, i.e.,
\begin{align*}
\int_{S^{n-1}}\varphi^1\mathcal{D}\varphi^2=\int_{S^{n-1}}\varphi^2\mathcal{D}\varphi^1.
\end{align*}
Indeed, $\mathcal{D}_1$ is self-adjoint obviously. In addition, according to the conclusion of \cite[last line of page 69]{HY0}, we know that $\mathcal{D}_2$ is self-adjoint.

By the $q$-homogeneity of $l(u)=|\nabla u|^q$, it yields that
\begin{align*}
\mathcal{D}(\rho^p_{\langle\rho_0\rangle^*})=q|\nabla u|^q.
\end{align*}
Hence, based on the above calculations, we get
\begin{align*}
&\lim_{s\rightarrow 0}\frac{\widetilde{Q}_{q,n-p}(\langle\rho_s\rangle^*)-\widetilde{Q}_{q,n-p}(\langle\rho_0\rangle^*)}{s}=\frac{d}{ds}\widetilde{Q}_{q,n-p}(\langle\rho_s\rangle^*)\bigg|_{s=0}\\
=&\frac{b}{a}\int_{S^{n-1}\backslash \eta_0}\bigg(-p\rho^{p}_{\langle\rho_0\rangle^*}(v)g(\alpha^*_{\langle\rho_0\rangle^*}(v))|\nabla u(r_{\langle\rho_0\rangle^*}(v))|^q+\rho^p_{\langle\rho_0\rangle^*}(v)\mathcal{D}(-pg(\alpha^*_{\langle\rho_0\rangle}(v))\rho^p_{\langle\rho_0\rangle^*}(v))\bigg)dv\\
=&\frac{b}{a}\int_{S^{n-1}\backslash \eta_0}\bigg(-p\rho^{p}_{\langle\rho_0\rangle^*}(v)g(\alpha^*_{\langle\rho_0\rangle^*}(v))|\nabla u(r_{\langle\rho_0\rangle^*}(v))|^q-pg(\alpha^*_{\langle\rho_0\rangle}(v))\rho^p_{\langle\rho_0\rangle^*}\mathcal{D}(\rho^p_{\langle\rho_0\rangle^*}(v))\bigg)dv\\
=&\frac{b}{a}\int_{S^{n-1}\backslash \eta_0}\bigg(-p\rho^{p}_{\langle\rho_0\rangle^*}(v)g(\alpha^*_{\langle\rho_0\rangle^*}(v))|\nabla u(r_{\langle\rho_0\rangle^*}(v))|^q-pqg(\alpha^*_{\langle\rho_0\rangle}(v))\rho^p_{\langle\rho_0\rangle^*}|\nabla u(r_{\langle\rho_0\rangle^*}(v))|^q\bigg)dv\\
=&\frac{-p(1+q)b}{a}\int_{S^{n-1}\backslash \eta_0}\rho^{p}_{\langle\rho_0\rangle^*}(v)g(\alpha^*_{\langle\rho_0\rangle^*}(v))|\nabla u(r_{\langle\rho_0\rangle^*}(v))|^qdv\\
=&\frac{-p(1+q)b}{a}\int_{S^{n-1}\setminus\eta_0}(\hat{g}\bm1_\eta)(\alpha_{\langle\rho_0\rangle^*}(v))\rho^p_{\langle\rho_0\rangle^*}(v)|\nabla u(r_{\langle\rho_0\rangle^*}(v))|^qdv\\
=&-p(1+q)\int_{S^{n-1}\setminus\eta_0}(\hat{g}\bm1_\eta)(\xi)d\widetilde{Q}_{q,n-p}(\langle\rho_0\rangle^*,\xi)\\
=&-p(1+q)\int_\eta g(\xi)d\widetilde{Q}_{q,n-p}(\langle\rho_0\rangle^*,\xi).
\end{align*}
Here, $g(\alpha_{\langle\rho_0\rangle^*}(v))=(\hat{g}\bm1_\eta)(\alpha_{\langle\rho_0\rangle^*}(v))$, it has been proven that $g$ can be extended to a continuous function $\hat{g}:S^{n-1}\rightarrow\mathbb{R}$, (see \cite[page 364]{HY}) for all $v\in S^{n-1}\setminus \eta_0$.
\end{proof}
\begin{theorem}\label{the390}
Let $\Omega\in\mathcal{K}^n_o$ and $f:S^{n-1}\rightarrow\mathbb{R}$ be continuous. If $[h_s]$ is a logarithmic family of the Wulff shapes with respect to $(h_\Omega, f)$, then, for $p\neq 0$ and $q>1$,
\begin{align*}
\lim_{s\rightarrow 0}\frac{\widetilde{Q}_{q,n-p}([h_s])-\widetilde{Q}_{q,n-p}(\Omega)}{s}=p(1+q)\int_{S^{n-1}}f(\xi)d\widetilde{Q}_{q,n-p}(\Omega,\xi).
\end{align*}
\end{theorem}
\begin{proof}
From the definition of $p$-th dual $q$-torsional rigidity (\ref{eq305}) and Theorem \ref{the39}, we attain
\begin{align*}
&\lim_{s\rightarrow 0}\frac{\widetilde{Q}_{q,n-p}([h_s])-\widetilde{Q}_{q,n-p}(\Omega)}{s}=\frac{d}{ds}\widetilde{Q}_{q,n-p}([h_s])\bigg|_{s=0}\\
=&\frac{b}{a}\int_{S^{n-1}}\bigg(\frac{d}{ds}\rho^p_{[h_s]}(v)\bigg|_{s=0}|\nabla u(r_\Omega(v))|^q+\rho^p_\Omega(v)\frac{d}{ds}|\nabla u(r_{[h_s]}(v))|^q\bigg|_{s=0}\bigg)dv\\
=&\frac{b}{a}\int_{S^{n-1}}\bigg(\lim_{s\rightarrow0}\frac{\rho^p_{[h_s]}(v)-\rho^p_\Omega(v)}{s}|\nabla u(r_\Omega(v))|^q+\rho^p_\Omega(v)\frac{d}{ds}|\nabla u(r_{[h_s]}(v))|^q\bigg|_{s=0}\bigg)dv\\
=&\frac{p(1+q)b}{a}\int_{S^{n-1}}f(\alpha_\Omega(v))\rho^p_\Omega(v)|\nabla u(r_\Omega(v))|^qdv\\
=&p(1+q)\int_{S^{n-1}}f(\xi)d\widetilde{Q}_{q,n-p}(\Omega,\xi).
\end{align*}

Here, the last second equality used Theorem \ref{the39}. For the convenience of readers, we give a simple explanation. The logarithmic family of Wulff shapes $[h_s]$ is defined as the Wulff shape of $h_s$, where $h_s$ is given by
\begin{align*}
\log h_s=\log h_\Omega+sf+o(s,\cdot).
\end{align*}
This and $\frac{1}{h_\Omega}=\rho_\Omega^*$, allow us to define
\begin{align*}
\log \rho_s^*=\log \rho_\Omega^*-sf-o(s,\cdot),
\end{align*}
and $\rho_s^*$ will generate a logarithmic family of convex hull $\langle\Omega^*,-f,-o,s\rangle$. From \cite[Lemma 2.8]{HY}, we know that $\langle\rho_s\rangle^*=[h_s]$ and $\langle\rho_0\rangle^*=[h_0]$, thus
\begin{align*}
[\Omega,f,o,s]=\langle\Omega^*,-f,-o,s\rangle^*.
\end{align*}
The desired result follows directly from Theorem \ref{the39}.
\end{proof}
\begin{corollary}\label{cor391}
Let $\Omega_1, \Omega_2\in\mathcal{K}_o^n$, $p\neq 0$ and $q>1$. Then
\begin{align*}
\lim_{s\rightarrow 0}\frac{\widetilde{Q}_{q,n-p}((1-s)\Omega_1+s\Omega_2)-\widetilde{Q}_{q,n-p}(\Omega_1)}{s}=p[\widetilde{Q}(\Omega_1,\Omega_2)-\widetilde{Q}_{q,n-p}(\Omega_1)],
\end{align*}
and
\begin{align*}
\lim_{s\rightarrow 0}\frac{\widetilde{Q}_{q,n-p}((1-s)\Omega_1+_0s\Omega_2)-\widetilde{Q}_{q,n-p}(\Omega_1)}{s}=p\int_{S^{n-1}}\log\bigg(\frac{h_{\Omega_2}(\xi)}{h_{\Omega_1}(\xi)}\bigg)
d\widetilde{Q}_{q,n-p}(\Omega_1,\xi).
\end{align*}
Here, $\widetilde{Q}(\Omega_1,\Omega_2)=\int_{S^{n-1}}\frac{h_{\Omega_2}(v)}{h_{\Omega_1}(v)}d\widetilde{Q}_{q,n-p}(\Omega_1,\xi)$.
\end{corollary}
\begin{proof}
For sufficiently small $s$, we define $h_s$ by
\begin{align*}
h_s=(1-s)h_{\Omega_1}+sh_{\Omega_2}=h_{\Omega_1}+s(h_{\Omega_2}-h_{\Omega_1}),
\end{align*}
taking the logarithm of both sides of the above equality, we obtain the following form:
\begin{align*}
\log h_s=\log h_{\Omega_1}+s\bigg(\frac{h_{\Omega_2}-h_{\Omega_1}}{h_{\Omega_1}}\bigg)+o(s,\cdot).
\end{align*}
From Theorem \ref{the390}, we get
\begin{align*}
&\lim_{s\rightarrow 0}\frac{\widetilde{Q}_{q,n-p}((1-s)\Omega_1+s\Omega_2)-\widetilde{Q}_{q,n-p}(\Omega_1)}{s}\\
=&p(1+q)\int_{S^{n-1}}\frac{h_{\Omega_2}-h_{\Omega_1}}{h_{\Omega_1}}d\widetilde{Q}_{q,n-p}(\Omega_1,\xi)\\
=&p(1+q)\int_{S^{n-1}}\frac{h_{\Omega_2}}{h_{\Omega_1}}d\widetilde{Q}_{q,n-p}(\Omega_1,\xi)-p(1+q)\int_{S^{n-1}}d\widetilde{Q}_{q,n-p}(\Omega_1,\xi)\\
=&p(1+q)[\widetilde{Q}(\Omega_1,\Omega_2)-\widetilde{Q}_{q,n-p}(\Omega_1)].
\end{align*}
Similarly, for sufficiently small $s$, we can also denote $h_s$ by
\begin{align*}
h_s=h_{\Omega_1}^{1-s}h_{\Omega_2}^s=h_{\Omega_1}\bigg(\frac{h_{\Omega_2}}{h_{\Omega_1}}\bigg)^s,
\end{align*}
then,
\begin{align*}
\log h_s=\log h_{\Omega_1}+s\bigg(\frac{h_{\Omega_2}}{h_{\Omega_1}}\bigg).
\end{align*}
Thus, we have following result by Theorem \ref{the390},
\begin{align*}
\lim_{s\rightarrow 0}\frac{\widetilde{Q}_{q,n-p}((1-s)\Omega_1+_0s\Omega_2)-\widetilde{Q}_{q,n-p}(\Omega_1)}{s}=&p(1+q)\int_{S^{n-1}}\log \frac{h_{\Omega_2}}{h_{\Omega_1}}d\widetilde{Q}_{q,n-p}(\Omega_1,\xi).
\end{align*}
\end{proof}
\begin{corollary}\label{cor392}
Let $\Omega_1, \Omega_2, \Omega_3\in\mathcal{K}_o^n$, $p\neq 0$ and $q>1$. Then
\begin{align*}
&\lim_{s\rightarrow 0}\frac{\widetilde{Q}_{q,p}((1-s)\Omega_1+_0s\Omega_2,\Omega_3)-\widetilde{Q}_{q,p}(\Omega_1,\Omega_3)}{s}\\
=&p(1+q)\int_{S^{n-1}}\log \frac{h_{\Omega_2}}{h_{\Omega_1}}d\widetilde{Q}_{q,p}(\Omega_1,\Omega_3,\xi),
\end{align*}
where
\begin{align*}
\widetilde{Q}_{q,p}(\Omega_1,\Omega_3,\eta)=\frac{b}{a}\int_{\alpha^*_{\Omega_1}(\eta)}\rho_{\Omega_1}(\xi)^p\rho_{\Omega_3}(\xi)^{n-p}|\nabla u(r_{\Omega_1}(\xi))|^qd\xi.
\end{align*}
\end{corollary}
\begin{proof} The result is directly obtained from formula (\ref{eq304}), Definition \ref{def34} and Corollary \ref{cor391}.
\end{proof}

\section{\bf Geometric flow and its associated functionals}\label{sec4}
In this subsection, we will introduce the geometric flow and its associated functionals to solve the $p$-th dual Minkowski problem for $q$-torsional rigidity with $q>1$. For convenience, the Gauss curvature flow is restated here. Let $\partial\Omega_0$ be a smooth, closed and origin-symmetric strictly convex hypersurface in $\mathbb{R}^n$ for $p<n$ ($p\neq 0$), $f$ is a positive smooth even function on $S^{n-1}$. Specially, for $p<0$, let $\partial\Omega_0$ be a smooth, closed strictly convex hypersurface in $\mathbb{R}^n$ containing the origin in its interior, $f$ is a positive smooth function on $S^{n-1}$. We consider the following Gauss curvature flow
\begin{align}\label{eq401}
\left\{
\begin{array}{lc}
\frac{\partial X(x,t)}{\partial t}=-\lambda(t)f(v)\frac{(|\nabla h|^2+h^2)^{\frac{n-p}{2}}}{|\nabla u(X,t)|^q}\mathcal{K}
(x,t)v+X(x,t),  \\
X(x,0)=X_0(x),\\
\end{array}
\right.
\end{align}
where $\mathcal{K}(x,t)$ is the Gauss curvature of hypersurface $\partial\Omega_t$ at $X(\cdot,t)$, $v=x$ is the unit outer
normal vector of $\partial\Omega_t$ at $X(\cdot,t)$, and $\lambda(t)$ is given by
\begin{align}\label{eq402}\lambda(t)=\frac{\int_{S^{n-1}}\rho^p|\nabla u|^qdv}{\int_{S^{n-1}}f(x)dx}=\frac{\frac{a}{b}\widetilde{Q}_{q,n-p}(\Omega_t)}{\int_{S^{n-1}}f(x)dx}.\end{align}

Taking the scalar product of both sides of the equation and of the initial condition in (\ref{eq401}) by $v$, by means of the definition of support function (\ref{eq203}) and formula (\ref{eq202}), we describe the flow (\ref{eq401}) with the support function as follows
\begin{align*}\left\{
\begin{array}{lc}
\frac{\partial h(x,t)}{\partial t}=-\lambda(t)f(x)\frac{(|\nabla h|^2+h^2)^{\frac{n-p}{2}}}{|\nabla u(\nabla h,t)|^q}\mathcal{K}
(x,t)+h(x,t),  \\
h(x,0)=h_0(x).\\
\end{array}
\right.\end{align*}
From this and $\rho^2=h^2+|\nabla h|^2$, we obtain
\begin{align}\label{eq403}\left\{
\begin{array}{lc}
\frac{\partial h(x,t)}{\partial t}=-\lambda(t)f(x)\frac{\rho(v,t)^{n-p}}{|\nabla u(\nabla h,t)|^q}\mathcal{K}
(x,t)+h(x,t),  \\
h(x,0)=h_0(x).\\
\end{array}
\right.\end{align}
Notice that
\begin{align}\label{eq404}
\frac{1}{\rho(v,t)}\frac{\partial\rho(v,t)}{\partial t}=\frac{1}{h(x,t)}\frac{\partial h(x,t)}{\partial t}.
\end{align}
Thus,
\begin{align}\label{eq405}\left\{
\begin{array}{lc}
\frac{\partial \rho(v,t)}{\partial t}=-\lambda(t)f(x)\frac{\rho(v,t)^{n-p+1}}{h(x,t)|\nabla u(\nabla h,t)|^q}\mathcal{K}
(v,t)+\rho(v,t),  \\
\rho(v,0)=\rho_0(v).\\
\end{array}
\right.\end{align}

Next, we discuss the characteristics of two essential geometric functionals with respect to Equation (\ref{eq403}) or (\ref{eq405}). Firstly, we show the $p$-th dual $q$-torsional rigidity unchanged along the flow (\ref{eq401}). In fact, the conclusion can be stated as the following lemma.

\begin{lemma}\label{lem41} Let $q>1$ and $p\neq 0$, then, the $p$-th dual $q$-torsional rigidity $\widetilde{Q}_{q,n-p}(\Omega_t)$ is unchanged with regard to Equation (\ref{eq405}) for $t\in[0,T)$, i.e.,
\begin{align*}\widetilde{Q}_{q,n-p}(\Omega_t)=\widetilde{Q}_{q,n-p}(\Omega_0).\end{align*}
Here, $T$ is the maximal time for existence of smooth solutions to the flow (\ref{eq401}).
\end{lemma}

\begin{proof}
Let $h(\cdot,t)$ and $\rho(\cdot,t)$ be the support function and the radial function of $\Omega_t$, respectively, $u(X,t)$ is the solution of (\ref{eq102}) in $\Omega_t$. From (\ref{eq305}) and Theorem \ref{the390}, we know that
\begin{align*}\frac{\partial}{\partial t}\widetilde{Q}_{q,n-p}(\Omega_t)=\frac{p(1+q)b}{a}\int_{S^{n-1}}\rho^{p-1}(v,t)\frac{\partial\rho}{\partial t}|\nabla u|^qdv.
\end{align*}
Thus, from (\ref{eq402}), (\ref{eq405}) and $\rho^n\mathcal{K}dv=hdx$, we have
\begin{align*}\frac{\partial}{\partial t}\widetilde{Q}_{q,n-p}(\Omega_t)=&\frac{p(1+q)b}{a}\int_{S^{n-1}}\rho^{p-1}(v,t)[-\lambda(t)f(x)\frac{\rho^{n-p+1}}{h(x,t)|\nabla u|^q}\mathcal{K}(v,t)+\rho]|\nabla u|^qdv\\
=&\frac{p(1+q)b}{a}\bigg[\int_{S^{n-1}}-\lambda(t)f(x)\frac{\rho^n\mathcal{K}(v,t)}{h(x,t)}dv+\int_{S^{n-1}}\rho^p|\nabla u|^qdv\bigg]\\
=&\frac{p(1+q)b}{a}\bigg[-\frac{\int_{S^{n-1}}\rho^p|\nabla u|^qdv}{\int_{S^{n-1}}f(x)dx}\int_{S^{n-1}}f(x)dx+\int_{S^{n-1}}\rho^p|\nabla u|^qdv\bigg]\\
=&0.
\end{align*}
This ends the proof of Lemma \ref{lem41}.
\end{proof}

The next lemma will show that the functional (\ref{eq110}) is non-increasing along the flow (\ref{eq401}).
\begin{lemma}\label{lem42}
The functional (\ref{eq110}) is non-increasing along the flow (\ref{eq401}). Namely, $\frac{\partial}{\partial t}\Phi(\Omega_t)\leq0$, the equality holds if and only if the support function of $\Omega_t$ satisfies (\ref{eq108}).
\end{lemma}
\begin{proof}
By (\ref{eq110}), (\ref{eq402}), (\ref{eq403}), (\ref{eq404}), $\rho^n\mathcal{K}dv=hdx$ and Theorem \ref{the390}, we obtain the following result for $p\neq 0$,
\begin{align*}
\frac{\partial}{\partial t}\Phi(\Omega_t)=&\frac{\int_{S^{n-1}}\frac{f(x)}{h}\frac{\partial h}{\partial t}dx}{\int_{S^{n-1}}f(x)dx}-\frac{\int_{S^{n-1}}\rho^{p-1}|\nabla u|^q\frac{\partial\rho}{\partial t}dv}{\int_{S^{n-1}}\rho^p|\nabla u|^qdv}\\
=&\frac{\int_{S^{n-1}}\frac{f(x)}{h}\frac{\partial h}{\partial t}dx}{\int_{S^{n-1}}f(x)dx}-\frac{\int_{S^{n-1}}\rho^{p-1}|\nabla u|^q\frac{\rho}{h}\frac{\partial h}{\partial t}\frac{h}{\rho^n\mathcal{K}}dx}{\int_{S^{n-1}}\rho^p|\nabla u|^qdv}\\
=&\int_{S^{n-1}}\frac{\partial h}{\partial t}\bigg[\frac{\frac{f(x)}{h}}{\int_{S^{n-1}}f(x)dx}-\frac{|\nabla u|^q}{\int_{S^{n-1}}\rho^p|\nabla u|^qdv}\frac{\rho^p}{h}\frac{h}{\rho^n\mathcal{K}}\bigg]dx\\
=&\int_{S^{n-1}}\frac{\partial h}{\partial t}\bigg[\frac{\frac{f(x)}{h}h\rho^{n-p}\mathcal{K}\int_{S^{n-1}}\rho^p|\nabla u|^qdv-|\nabla u|^qh\int_{S^{n-1}}f(x)dx}{h\rho^{n-p}\mathcal{K}\int_{S^{n-1}}f(x)dx\int_{S^{n-1}}\rho^p|\nabla u|^qdv}\bigg]dx\\
=&\int_{S^{n-1}}\frac{\partial h}{\partial t}\bigg[\frac{\frac{f(x)\rho^{n-p}\mathcal{K}}{|\nabla u|^q}\frac{\int_{S^{n-1}}\rho^p|\nabla u|^qdv}{\int_{S^{n-1}}f(x)dx}-h}{\frac{h\rho^{n-p}\mathcal{K}}{|\nabla u|^q}\int_{S^{n-1}}\rho^p|\nabla u|^qdv}\bigg]dx\\
=&-\int_{S^{n-1}}\frac{[-\lambda(t)\frac{f(x)\rho^{n-p}\mathcal{K}}{|\nabla u|^q}+h]^2}{\frac{h\rho^{n-p}\mathcal{K}}{|\nabla u|^q}\int_{S^{n-1}}\rho^p|\nabla u|^qdv}dx\\
\leq &0.
\end{align*}
Above equality holds if and only if $\lambda(t)\frac{f(x)\rho^{n-p}\mathcal{K}}{|\nabla u|^q}=h$, i.e.,
\begin{align*}
\frac{1}{\lambda(t)}|\nabla u|^qh\times(\rho^2)^\frac{{p-n}}{2}\det(\nabla_{ij}h+h\delta_{ij})=f,
\end{align*}
equivalently,
\begin{align*}
\frac{1}{\lambda(t)}|\nabla u|^qh\times(h^2+|\nabla h|^2)^\frac{{p-n}}{2}\det(\nabla_{ij}h+h\delta_{ij})=f.
\end{align*}
Namely, the support function of $\Omega_t$ satisfies (\ref{eq108}).
\end{proof}

\section{\bf Priori estimates}\label{sec5}

In this subsection, we establish the $C^0$, $C^1$ and $C^2$ estimates for solutions to Equation (\ref{eq403}). In the following of this paper, we always assume that  $\partial\Omega_0$ be a smooth, closed and origin-symmetric strictly convex hypersurface in $\mathbb{R}^n$ for $p<n$ ($p\neq 0$). Specially, for $p<0$, let $\partial\Omega_0$ be a smooth, closed and strictly convex hypersurface in $\mathbb{R}^n$ containing the origin in its interior. $h:S^{n-1}\times [0,T)\rightarrow \mathbb{R}$ is a smooth solution to Equation (\ref{eq403}) with the initial $h(\cdot,0)$ the support function of $\partial\Omega_0$. Here, $T$ is the maximal time for existence of smooth solutions to Equation (\ref{eq403}).

\subsection{$C^0$, $C^1$ estimates}

In order to complete the $C^0$ estimate, we firstly need to introduce the following lemmas for convex bodies.
\begin{lemma}\label{lem51}\cite[Lemma 2.6]{CH}
Let $\Omega\in\mathcal{K}^n_o$, $h$ and $\rho$ are respectively the support function and the radial function of $\Omega$, and $x_{\max}$ and $\xi_{\min}$ are two points such that
$h(x_{\max})=\max_{S^{n-1}}h$ and $\rho(\xi_{\min})=\min_{S^{n-1}}\rho$. Then,
\begin{align*}
\max_{S^{n-1}}h=&\max_{S^{n-1}}\rho \quad \text{and} \quad \min_{S^{n-1}}h=\min_{S^{n-1}}\rho;\end{align*}
\begin{align*}h(x)\geq& x\cdot x_{\max}h(x_{\max}),\quad \forall x\in S^{n-1};\end{align*}
\begin{align*}\rho(\xi)\xi\cdot\xi_{\min}\geq&\rho(\xi_{\min}),\quad \forall \xi\in S^{n-1}.\end{align*}
\end{lemma}

\begin{remark}
The results in Lemma \ref{lem51} are also tenable for any $t\geq 0$, for example, we can write
\begin{align*}h(x,t)\geq& x\cdot x^t_{\max}h(x_{\max},t),\quad \forall x\in S^{n-1}.\end{align*}
\end{remark}

\begin{lemma}\label{lem52} Let $u\in W_{loc}^{1,q}(\Omega)$ be a local weak solution of
\begin{align*}
{\rm div}(|\nabla u|^{q-2}\nabla u)=\psi, ~~ q>1; ~~ \psi\in L_m^{loc}(\Omega),
\end{align*}
$m>q^\prime n$ $(\frac{1}{q^\prime}+\frac{1}{q}=1)$. Then, $u\in C^{1+\alpha}_{loc}(\Omega)$. (see \cite[Corollary in pp. 830]{DE})
\end{lemma}

\begin{lemma}\label{lem53}
Let $\partial\Omega_t$ be a smooth even convex solution to the flow (\ref{eq401}) with $p<n$ ($p\neq 0$) in $\mathbb{R}^n$, $u(X,t)$ is the solution of (\ref{eq102}) in $\Omega_t$, $f$ is a positive smooth even function on $S^{n-1}$. Then, there is a positive constant $C$ being independent of
$t$ such that
\begin{align}\label{eq501}
\frac{1}{C}\leq h(x,t)\leq C, \ \ \forall(x,t)\in S^{n-1}\times[0,T),
\end{align}
\begin{align}\label{eq502}
\frac{1}{C}\leq \rho(v,t)\leq C, \ \ \forall(v,t)\in S^{n-1}\times[0,T).
\end{align}
Here, $h(x,t)$ and $\rho(v,t)$ are the support function and the radial function of $\Omega_t$, respectively.
\end{lemma}
\begin{proof}
Due to  $\rho(v,t)v=\nabla h(x,t)+h(x,t)x$. Clearly, one sees
\begin{align*}
\min\limits_{S^{n-1}} h(x,t)\leq \rho (v,t)\leq \max\limits_{S^{n-1}} h(x,t).
\end{align*}
This implies that the estimate (\ref{eq501}) is tantamount to the estimate (\ref{eq502}). Hence, we only need to establish (\ref{eq501}) or (\ref{eq502}).
Using the unchanged property of $\widetilde{Q}_{q,n-p}(\Omega_t)$ and the monotonicity of $\Phi(\Omega_t)$, we have
\begin{align*}
\Phi(\Omega_0)\geq\Phi(\Omega_t)=&\frac{\int_{S^{n-1}}\log h(x,t)f(x)dx}{\int_{S^{n-1}}f(x)dx}-\log\bigg(\int_{S^{n-1}}\rho^p(v,t)|\nabla u|^qdv\bigg)^\frac{1}{p(1+q)}\\
=&\frac{\int_{S^{n-1}}\log h(x,t)f(x)dx}{\int_{S^{n-1}}f(x)dx}-\log\bigg(\frac{a}{b}\widetilde{Q}_{q,n-p}(\Omega_t)\bigg)^{\frac{1}{p(1+q)}}\\
=&\frac{\int_{S^{n-1}}\log h(x,t)f(x)dx}{\int_{S^{n-1}}f(x)dx}-\log\bigg(\frac{a}{b}\widetilde{Q}_{q,n-p}(\Omega_0)\bigg)^{\frac{1}{p(1+q)}}.
\end{align*}
Thus,
\begin{align*}
&\bigg[\Phi(\Omega_0)+\log\bigg(\frac{a}{b}\widetilde{Q}_{q,n-p}(\Omega_0)\bigg)^{\frac{1}{p(1+q)}}\bigg]\int_{S^{n-1}}f(x)dx\\
\geq&\int_{S^{n-1}}\log h(x,t)f(x)dx\\
\geq &\int_{S^{n-1}}f(x)\log[h(x_{\max},t)x\cdot x^t_{\max}]dx\\
\geq&\log h(x_{\max},t)\int_{S^{n-1}}f(x)dx+\int_{\{x\in S^{n-1}: x\cdot x^t_{\max}\geq\frac{1}{2}\}}f(x)\log(x\cdot x^t_{\max})dx\\
\geq&C\log h(x_{\max},t)-c\int_{\{x\in S^{n-1}: x\cdot x^t_{\max}\geq\frac{1}{2}\}}f(x)dx\\
\geq&C\log h(x_{\max},t)-c_1.
\end{align*}
This yields
\begin{align*}
\log h(x_{\max},t)\leq \bigg[\frac{\bigg(\Phi(\Omega_0)+\log\bigg(\frac{a}{b}\widetilde{Q}_{q,n-p}(\Omega_0)\bigg)^{\frac{1}{p(1+q)}}\bigg)\int_{S^{n-1}}f(x)dx+c_1}{C}\bigg].
\end{align*}
Thus,
\begin{align*}
\sup h(x_{\max},t)\leq \exp\bigg[\frac{\bigg(\Phi(\Omega_0)+\log\bigg(\frac{a}{b}\widetilde{Q}_{q,n-p}(\Omega_0)\bigg)^{\frac{1}{p(1+q)}}\bigg)\int_{S^{n-1}}f(x)dx+c_1}{C}\bigg]\leq C_1,
\end{align*}
where $C, C_1, c_1$ are positive constants being independent of $t$.

To prove the lower bound of $h(x,t)$, we use contradiction. Let us suppose that $\{t_k\}\subset [0,T)$ is a sequence such that $h(x,t_k)$ is not uniformly bounded away from $0$, i.e., $\min_{S^{n-1}}h(x,t_k)\rightarrow 0$ as $k \rightarrow \infty$. On the other hand, making use of the upper bound, by the Blaschke-Selection theorem \cite{SC}, there is a subsequence in $\{\Omega_{t_k}\}$, for convenience, which is still denoted by $\{\Omega_{t_k}\}$, such that $\{\Omega_{t_k}\}\rightarrow \widetilde{\Omega}$ as $k\rightarrow\infty$, where $\widetilde{\Omega}$ is a origin-symmetric convex body. Then, we obtain $\min_{S^{n-1}}h(\widetilde{\Omega},\cdot)=\lim_{k\rightarrow \infty}\min_{S^{n-1}}h(\Omega_{t_k},\cdot)=0$. This implies that $\widetilde{\Omega}$ is contained in a lower-dimensional subspace in $\mathbb{R}^n$. This can lead to $\rho(\Omega_{t_k},\cdot)\rightarrow 0$ as $k\rightarrow\infty$ almost everywhere with respect to the spherical Lebesgue measure. According to bounded convergence theorem, we can derive
\begin{align*}\widetilde{Q}_{q,n-p}(\widetilde{\Omega})=&\frac{b}{a}\int_{S^{n-1}}\rho^p(\widetilde{\Omega},\cdot)|\nabla u|^qdv\\
=&\lim_{k\rightarrow\infty}\frac{b}{a}\int_{S^{n-1}}\rho^p(\Omega_{t_k},\cdot)|\nabla u|^qdv\rightarrow 0.
\end{align*}
However, Lemma \ref{lem41} shows that
\begin{align*}
\widetilde{Q}_{q,n-p}(\widetilde{\Omega})= \widetilde{Q}_{q,n-p}(\Omega_0)=\overline{C}_0~~\text{(positive constant)}\neq 0,
\end{align*}
which is a contradiction. It
follows that $h(x, t)$ has an uniform lower bound. Therefore, we complete the estimate of Lemma \ref{lem53}.
\end{proof}

Specially, for $p<0$, we can also establish the following estimates.

\begin{lemma}\label{lem54}
Let $\partial\Omega_t$ be a smooth non-even convex solution to the flow (\ref{eq401}) with $p<0$ in $\mathbb{R}^n$, $u(X,t)$ is the solution of (\ref{eq102}) in $\Omega_t$, $f$ is a positive smooth function on $S^{n-1}$. Then, there is a positive constant $C$ being independent of $t$ such that
\begin{align}\label{eq503}
\frac{1}{C}\leq h(x,t)\leq C, \ \ \forall(x,t)\in S^{n-1}\times[0,T),
\end{align}
\begin{align}\label{eq504}
\frac{1}{C}\leq \rho(v,t)\leq C, \ \ \forall(v,t)\in S^{n-1}\times[0,T).
\end{align}
Here, $h(x,t)$ and $\rho(v,t)$ are the support function and the radial function of $\Omega_t$, respectively.
\end{lemma}
\begin{proof}

Similar to the proof of Lemma \ref{lem53}, we only estimate (\ref{eq503}) or (\ref{eq504}). Before proving Lemma \ref{lem54}, let's first explain the following facts. From Lemma \ref{lem52}, we know that $\nabla u$ is H\"{o}lder continuous, thus, it is not difficult to obtain $|\nabla u(X(x,t),t)|\neq 0$ from Lemma \ref{lem41} for any $t\in[0,T)$. Note that since $\Omega_t$ is a convex body (a compact and convex subset of $\mathbb{R}^n$ with nonempty interior), we know that $|\nabla u (X(x,t),t)|$ has a positive upper bound and a positive lower bound on $\Omega_t$. Further, according to (\ref{eq202}), we obtain $|\nabla u (\nabla h(x,t),t)|$ has a positive upper bound and a positive lower bound on $S^{n-1}\times [0,T)$.

In the same time, by virtue of Schauder's theory (see example Chapter 6 in \cite{GI}), $|\nabla^k u (\nabla h(x,t),t)|$ is bounded on $S^{n-1}\times [0,T)$, for all integer $k\geq 2$.

Moreover, since $f$ is a positive smooth function on $S^{n-1}$, and combining the conditions of initial hypersurface $\partial\Omega_0$, Lemma \ref{lem41} shows that
\begin{align*}
\lambda(t)=\frac{\frac{a}{b}\widetilde{Q}_{q,n-p}(\Omega_t)}{\int_{S^{n-1}}f(x)dx}=\frac{\frac{a}{b}\widetilde{Q}_{q,n-p}(\Omega_0)}{\int_{S^{n-1}}f(x)dx}
\end{align*}
is a positive constant, we denote it $\overline{C}$.

Now, we firstly estimate $\min_{S^{n-1}}h(x,t)= h_{\min}$. Suppose $\min_{S^{n-1}}h(x,t)$ is attained at $x_0$. Then, at point $(x_0,t)$,
\begin{align*}
\nabla h=0, \quad \nabla^2 h\geq 0, \quad \rho=h,
\end{align*}
and
\begin{align*}
\nabla^2 h+hI\geq hI.
\end{align*}
Hence, at $(x_0,t)$,
\begin{align*}
\frac{\partial h(x,t)}{\partial t}=&-\lambda(t)f(x)\frac{\rho(v,t)^{n-p}}{|\nabla u(\nabla h(x,t),t)|^q}\mathcal{K}(x,t)+h(x,t)\\
\geq&-\overline{C}\max\limits_{S^{n-1}} f(x)\frac{h_{\min}^{n-p}}{|\nabla u(\nabla h,t)|^qh_{\min}^{n-1}}+h_{\min}\\
=&-\frac{\overline{C}\max\limits_{S^{n-1}} f(x)h_{\min}^{1-p}}{|\nabla u(\nabla h,t)|^q}+h_{\min}\\
=&\frac{\overline{C}\max\limits_{S^{n-1}} f(x)h_{\min}}{|\nabla u(\nabla h,t)|^q}\bigg[-\bigg(h_{\min}^{-p}-\frac{|\nabla u(\nabla h,t)|^q}{\overline{C}\max\limits_{S^{n-1}} f(x)}\bigg)\bigg].
\end{align*}

If $h_{\min}^{-p}\leq\frac{|\nabla u(\nabla h,t)|^q}{\overline{C}\max\limits_{S^{n-1}} f(x)}$, we obtain $\frac{\partial h(x_0,t)}{\partial t}\geq 0$. This implies
\begin{align*}
h_{\min}(x,t)\geq h(x_0,0).
\end{align*}

If $h_{\min}^{-p}>\frac{|\nabla u(\nabla h,t)|^q}{\overline{C}\max\limits_{S^{n-1}} f(x)}$, i.e., $\frac{1}{h_{\min}^p}>\frac{|\nabla u(\nabla h,t)|^q}{\overline{C}\max\limits_{S^{n-1}} f(x)}$, then $h_{\min}^p<\frac{\overline{C}\max\limits_{S^{n-1}} f(x)}{|\nabla u(\nabla h,t)|^q}$. Thus, for $p<0$, we have
\begin{align*}
h_{\min}(x,t)>\bigg(\frac{\overline{C}\max\limits_{S^{n-1}} f(x)}{|\nabla u(\nabla h,t)|^q}\bigg)^{\frac{1}{p}}.
\end{align*}
Combining the above two situations, we obtain
\begin{align*}
h_{\min}(x,t)\geq \min\bigg\{h(x_0,0),\bigg(\frac{\overline{C}\max\limits_{S^{n-1}} f(x)}{|\nabla u(\nabla h,t)|^q}\bigg)^{\frac{1}{p}}\bigg\},
\end{align*}
and above explanations indicate that $|\nabla u(\nabla h,t)|^q$ is bounded.

Next, we estimate $\max_{S^{n-1}}h(x,t)= h_{\max}$. Suppose $\max_{S^{n-1}}h(x,t)$ is obtained at $x_1$. Then, at point $(x_1,t)$,
\begin{align*}
\nabla h=0, \quad \nabla^2 h\leq 0, \quad \rho=h,
\end{align*}
and
\begin{align*}
\nabla^2 h+hI\leq hI.
\end{align*}
Thus, at $(x_1,t)$,
\begin{align*}
\frac{\partial h(x_1,t)}{\partial t}=&-\lambda(t)f(x)\frac{\rho(v,t)^{n-p}}{|\nabla u(\nabla h(x,t),t)|^q}\mathcal{K}(x,t)+h(x,t)\\
\leq&-\overline{C}\min\limits_{S^{n-1}} f(x)\frac{h_{\max}^{n-p}}{|\nabla u(\nabla h,t)|^qh_{\max}^{n-1}}+h_{\max}\\
=&-\frac{\overline{C}\min\limits_{S^{n-1}} f(x)h_{\max}^{1-p}}{|\nabla u(\nabla h,t)|^q}+h_{\max}\\
=&\frac{\overline{C}\min\limits_{S^{n-1}} f(x)h_{\max}}{|\nabla u(\nabla h,t)|^q}\bigg[-\bigg(h_{\max}^{-p}-\frac{|\nabla u(\nabla h,t)|^q}{\overline{C}\min\limits_{S^{n-1}} f(x)}\bigg)\bigg].
\end{align*}

If $h_{\max}^{-p}\geq\frac{|\nabla u(\nabla h,t)|^q}{\overline{C}\min\limits_{S^{n-1}} f(x)}$, we obtain $\frac{\partial h(x_1,t)}{\partial t}\leq 0$. This implies
\begin{align*}
h_{\max}(x,t)\leq h(x_1,0).
\end{align*}

If $h_{\max}^{-p}<\frac{|\nabla u(\nabla h,t)|^q}{\overline{C}\min\limits_{S^{n-1}} f(x)}$, i.e., $\frac{1}{h_{\max}^{p}}<\frac{|\nabla u(\nabla h,t)|^q}{\overline{C}\min\limits_{S^{n-1}} f(x)}$, then $h_{\max}^p>\frac{\overline{C}\min\limits_{S^{n-1}} f(x)}{|\nabla u(\nabla h,t)|^q}$. Therefore, for $p<0$, we have
\begin{align*}
h_{\max}(x,t)<\bigg(\frac{\overline{C}\min\limits_{S^{n-1}} f(x)}{|\nabla u(\nabla h,t)|^q}\bigg)^{\frac{1}{p}}.
\end{align*}
Combining the above two cases, we can obtain
\begin{align*}
h_{\max}(x,t)\leq \max\bigg\{h(x_1,0),\bigg(\frac{\overline{C}\min\limits_{S^{n-1}} f(x)}{|\nabla u(\nabla h,t)|^q}\bigg)^{\frac{1}{p}}\bigg\}.
\end{align*}
\end{proof}

\begin{lemma}\label{lem55}Let $\partial\Omega_t$ be a smooth, even convex solution to the flow (\ref{eq401}) for $p<n$ ($p\neq 0$) in $\mathbb{R}^n$, $f$ is a positive smooth  even function on $S^{n-1}$. Then, there is a positive constant $C$ being independent of $t$ such that
\begin{align}\label{eq505}|\nabla h(x,t)|\leq C,\quad\forall(x,t)\in S^{n-1}\times [0,T),\end{align}
and
\begin{align}\label{eq506}|\nabla \rho(v,t)|\leq C,\quad \forall(v,t)\in S^{n-1}\times [0,T).\end{align}

Specially, for $p<0$, $\partial\Omega_t$ be a smooth, non-even convex solution to the flow (\ref{eq401}) in $\mathbb{R}^n$, $f$ is a positive smooth function on $S^{n-1}$. Then, the estimates (\ref{eq505}) and (\ref{eq506}) are also established.
\end{lemma}

\begin{proof}
The desired results immediately follow from Lemma \ref{lem53} (or Lemma \ref{lem54}) and the following identities (see e.g. \cite{LR})
\begin{align*}
h=\frac{\rho^2}{\sqrt{\rho^2+|\nabla\rho|^2}},\qquad\rho^2=h^2+|\nabla h|^2.\end{align*}
\end{proof}

\subsection{$C^2$ estimate}

In this subsection, we establish the upper bound and the lower bound of principal curvature. This will show that Equation (\ref{eq403}) is uniformly parabolic. The technique used in this proof was first introduced by Tso \cite{TK} to derive the upper bound of Gauss curvature. Let us first give the following lemma before the $C^2$ estimate.

\begin{lemma}\label{lem56} Let $\Omega_t$ be a convex body of $C_+^2$ in $\mathbb{R}^n$ and $u(X(x,t),t)$ is the solution of (\ref{eq102}) with $q>1$ in $\Omega_t$, then
\begin{flalign*}
\begin{split}
(i)&(\nabla^2u(X(x,t),t)e_i)\cdot e_j=-\mathcal{K}|\nabla u(X(x,t),t)|c_{ij}(x,t);\\
(ii)&(\nabla^2u(X(x,t),t)e_i)\cdot x=-\mathcal{K}|\nabla u(X(x,t),t)|_jc_{ij}(x,t);\\
(iii)&(\nabla^2u(X(x,t),t)x)\cdot x=\frac{1}{q-1}\bigg(\mathcal{K}|\nabla u|{\rm Tr}(c_{ij}(h_{ij}+h\delta_{ij}))-|\nabla u|^{2-q}\bigg).
\end{split}&
\end{flalign*}
Here, $e_i$ and $x$ are orthonormal frame and unite outer normal on $S^{n-1}$, $\cdot$ is standard inner product and $c_{ij}$ is the cofactor matrix of $(h_{ij}+h\delta_{ij})$ with $\sum_{i,j}c_{ij}(h_{ij}+h\delta_{ij})=(n-1)\mathcal{K}^{-1}$.
\end{lemma}
\begin{proof} The conclusions similar to Lemma \ref{lem56} have been presented in some references, for example \cite{HJ}. Here, we will briefly state the proofs combining with our problem for the reader's convenience.

(i)~Let $h(x,t)$ be the support function of $\Omega_t$ for $(x,t)\in S^{n-1}\times [0,T)$ and $\iota=\frac{\partial h}{\partial t}$. Then, $X(x,t)=h_ie_i+hx, \frac{\partial X(x,t)}{\partial t}=\dot{X}(x,t)=\frac{\partial}{\partial t}(h_ie_i+hx)=\iota_ie_i+\iota x$. $X_i(x,t)=(h_{ij}+h\delta_{ij})e_j$, let $h_{ij}+h\delta_{ij}=\omega_{ij}$, then, $X_{ij}(x,t)=\omega_{ijk}e_k-\omega_{ij}x$, where $\omega_{ijk}$ is the covariant derivative of $\omega_{ij}$.

From $u(X,t)=0$ on $\partial\Omega_t$, it is not difficult for us to obtain
\begin{align*}\nabla u\cdot X_i=0,
\end{align*}
and
\begin{align*}((\nabla^2u)X_j)X_i+\nabla u X_{ij}=0.
\end{align*}
It follows that
\begin{align}\label{eq507}
\omega_{ik}\omega_{jl}(((\nabla^2u)e_l)\cdot e_k)+\omega_{ij}|\nabla u|=0.
\end{align}
Multiplying both sides of (\ref{eq507}) by $c_{ij}$, we have
\begin{align*}
c_{ij}\omega_{ik}\omega_{jl}(((\nabla^2 u)e_l)\cdot e_k)+\det(h_{ij}+h\delta_{ij})|\nabla u|=0.
\end{align*}
Namely,
\begin{align*}
\delta_{jk}\det(h_{ik}+h\delta_{ik})\omega_{jl}(((\nabla^2 u)e_l)\cdot e_k)+\det(h_{ij}+h\delta_{ij})|\nabla u|=0.
\end{align*}
This yields
\begin{align*}
\omega_{ij}(((\nabla^2 u)e_i)\cdot e_j)+|\nabla u|=0,
\end{align*}
then,
\begin{align*}
c_{ij}\omega_{ij}(((\nabla^2 u)e_i)\cdot e_j)+c_{ij}|\nabla u|=0,
\end{align*}
i.e.,
\begin{align*}
\mathcal{K}^{-1}(((\nabla^2 u)e_i)\cdot e_j)+c_{ij}|\nabla u|=0,
\end{align*}
thus,
\begin{align*}
((\nabla^2 u)e_i)\cdot e_j=-c_{ij}\mathcal{K}|\nabla u|.
\end{align*}
This completes the proof of (i).

(ii)~~Recall that
\begin{align*}
|\nabla u(X(x,t),t)|=-\nabla u(X(x,t),t)\cdot x,
\end{align*}
taking the covariant of both sides for above formula, we obtain
\begin{align}\label{eq508}
|\nabla u|_j=-\nabla u\cdot e_j-(\nabla^2 u)X_j\cdot x=-\omega_{ij}((\nabla^2 u)e_i\cdot x).
\end{align}
Multiplying both sides of (\ref{eq508}) by $c_{lj}$ and combining
\begin{align*}
c_{lj}\omega_{ij}=\delta_{li}\det(h_{ij}+h\delta_{ij}),
\end{align*}
we conclude that
\begin{align*}
c_{ij}|\nabla u|_j=-\det(h_{ij}+h\delta_{ij})(\nabla^2u)e_i\cdot x.
\end{align*}
Hence,
\begin{align*}
((\nabla^2u)e_i)\cdot x=-\mathcal{K}c_{ij}|\nabla u|_j.
\end{align*}
This proves (ii).

(iii)~~From (\ref{eq102}), we know that
\begin{align*}-1={\rm div}(|\nabla u|^{q-2}\nabla u)=|\nabla u|^{q-2}(\Delta u+\frac{q-2}{|\nabla u|^2}(\nabla^2u\nabla u)\cdot\nabla u),
\end{align*}
then,
\begin{align*}
\frac{q-2}{|\nabla u|^2}(\nabla^2 u\nabla u )\cdot\nabla u=-\Delta u-|\nabla u|^{2-q},
\end{align*}
further,
\begin{align*}
&(q-2)((\nabla^2 u) x) \cdot x \\
&=-\Delta u-|\nabla u|^{2-q}\\
&=-{\rm Tr}(\nabla^2u)-|\nabla u|^{2-q}\\
&=-\sum_{i}((\nabla^2u) e_i) \cdot e_j-((\nabla^2u)x) \cdot x-|\nabla u|^{2-q}\\
&=\mathcal{K}|\nabla u|{\rm Tr}(c_{ij}(h_{ij}+h\delta_{ij}))-((\nabla^2u)x)\cdot x-|\nabla u|^{2-q}.
\end{align*}
Hence,
\begin{align*}
(q-1)((\nabla^2u)x) \cdot x=\mathcal{K}|\nabla u|{\rm Tr}(c_{ij}(h_{ij}+h\delta_{ij}))-|\nabla u|^{2-q},
\end{align*}
consequently,
\begin{align*}
((\nabla^2u) x) \cdot x=\frac{1}{q-1}\bigg(\mathcal{K}|\nabla u|{\rm Tr}(c_{ij}(h_{ij}+h\delta_{ij}))-|\nabla u|^{2-q}\bigg).
\end{align*}
This provides the proof of (iii).
\end{proof}

By Lemma \ref{lem53} (or  Lemma \ref{lem54}) and Lemma \ref{lem55}, if $h$ is a smooth solution of Equation (\ref{eq403}) on $S^{n-1}\times [0,T)$ and $f$ is a positive smooth even function on $S^{n-1}$ for $p<n$ $(p\neq 0)$ (or $f$ is a positive smooth function on $S^{n-1}$ for $p<0$). Then, along the flow (\ref{eq401}) for $[0,T), \nabla h+hx$ and $h$ are smooth functions whose
ranges are within some bounded domain $\Omega_{[0,T)}$ and bounded interval $I_{[0,T)}$, respectively. Here, $\Omega_{[0,T)}$ and $I_{[0,T)}$ depend only
on the upper bound and the lower bound of $h$ on $[0,T)$.

\begin{lemma}\label{lem57}~~Let $q>1$, $\partial\Omega_t$ be a smooth, even convex solution to the flow (\ref{eq401}) for $p<n$ ($p\neq 0$) in $\mathbb{R}^n$, $f$ is a positive smooth  even function on $S^{n-1}$. Then, there is a positive constant $C$ being independent $t$ such that the principal curvatures $\kappa_i$ of $\partial\Omega_t$, $i=1,\cdots, n-1$, are bounded from above and below, satisfying
\begin{align*}\frac{1}{C}\leq \kappa_i(x,t)\leq C, \quad\forall (x,t)\in S^{n-1}\times [0,T),\end{align*}
where $C$ depends on $\|f\|_{C^0(S^{n-1})}$, $\|f\|_{C^1(S^{n-1})}$, $\|f\|_{C^2(S^{n-1})}$, $\|h\|_{C^0(S^{n-1}\times [0,T))}$, $\|h\|_{C^1(S^{n-1}\times [0,T))}$ and $\|\lambda\|_{C^0(S^{n-1}\times [0,T))}$.

Specially, let $p<0$, $\partial\Omega_t$ be a smooth, non-even convex solution to the flow (\ref{eq401}) in $\mathbb{R}^n$, $f$ is a positive smooth function on $S^{n-1}$. Then, the above conclusion is also tenable.
\end{lemma}

\begin{proof} The proof is divided into two parts: in the first part, we derive an upper bound for the Gauss curvature $\mathcal{K}(x,t)$; in the second part, we give an estimate of bound above for the principal radii $b_{ij}=h_{ij}+h\delta_{ij}$.

Step 1: Prove $\mathcal{K}\leq C$.

Firstly, we construct the following auxiliary function,
\begin{align}\label{eq509-}
M(x,t)=\frac{\lambda(t)f(x)\rho^{n-p}|\nabla u|^{-q}\mathcal{K}
-h}{h-\varepsilon_0}
\equiv\frac{-h_t}{h-\varepsilon_0},
\end{align}
where
\begin{align*}\varepsilon_0=\frac{1}{2}\min\limits_{S^{n-1}\times[0, T)}h(x,t)>0 \quad \text{and} \quad h_t =\frac{\partial h}{\partial t}.\end{align*}

For any fixed $t\in[0,T)$, we assume that $M(x_0, t)=\max\limits_{S^{n-1}}M(x,t)$ is the spatial maximum of $M$. Then, at $(x_0,t)$, we have
\begin{align}\label{eq509}
0=\nabla_iM=\frac{-h_{ti}}{h-\varepsilon_0}+\frac{h_th_i}{(h-\varepsilon_0)^2}.
\end{align}
Moreover, at $(x_0,t)$, we also get the following result from (\ref{eq509}),
\begin{align}\label{eq510}
\nonumber 0\geq\nabla_{ii}M=&\frac{-h_{tii}}{h-\varepsilon_0}+\frac{h_{ti}h_{i}}{(h-\varepsilon_0)^2}
+\frac{h_{ti}h_{i}+h_th_{ii}}{(h-\varepsilon_0)^2}-\frac{h_th_i(2(h-\varepsilon_0)h_i)}{(h-\varepsilon_0)^4}\\
\nonumber=&\frac{-h_{tii}}{h-\varepsilon_0}+\frac{2h_{ti}h_{i}+h_th_{ii}}
{(h-\varepsilon_0)^2}
-\frac{2h_th_ih_i}{(h-\varepsilon_0)^3}\\
\nonumber=&\frac{-h_{tii}}{h-\varepsilon_0}+\frac{h_th_{ii}}
{(h-\varepsilon_0)^2}
+\frac{2h_{ti}h_i(h-\varepsilon_0)-2h_th_ih_i}{(h-\varepsilon_0)^3}\\
=&\frac{-h_{tii}}{h-\varepsilon_0}+\frac{h_th_{ii}}
{(h-\varepsilon_0)^2}.
\end{align}
By (\ref{eq510}), we obtain
$$-h_{tii}\leq\frac{-h_th_{ii}}{h-\varepsilon_0},$$
hence,
\begin{align}\label{eq511}\nonumber-h_{tii}-h_t\delta_{ii}\leq&\frac{-h_th_{ii}}{h-\varepsilon_0}-h_t\delta_{ii}=\frac{-h_t}
{h-\varepsilon_0}(h_{ii}+(h-\varepsilon_0)\delta_{ii})\\
=&M(h_{ii}+h\delta_{ii}-\epsilon_0\delta_{ii})
=M(b_{ii}-\varepsilon_0\delta_{ii}).\end{align}

In addition, at $(x_0,t)$, we also have
\begin{align}\label{eq512}\frac{\partial}{\partial t}M=&\frac{-h_{tt}}{h-\epsilon_0}+\frac{h_t^2}{(h-\epsilon_0)^2}\\
\nonumber=&\frac{f}{h-\epsilon_0}\bigg[\frac{\partial(\lambda(t)\rho^{n-p}|\nabla u|^{-q})}{\partial t}\mathcal{K}+\lambda(t)\rho^{n-p}|\nabla u|^{-q}\frac{\partial(\det(\nabla^2h+hI))^{-1}}{\partial t}\bigg]+M+M^2,\end{align}
where
\begin{align*}
\frac{\partial}{\partial t}(\rho^{n-p}|\nabla u|^{-q})=(n-p)\rho^{n-p-1}\frac{\partial\rho}{\partial t}|\nabla u|^{-q}-q\rho^{n-p}|\nabla u|^{-q-1}\frac{\partial}{\partial t}|\nabla u|.
\end{align*}
From $\rho^2=h^2+|\nabla h|^2$ and $(x_0,t)$ is a maximum point of $M$, we get
\begin{align*}
\frac{\partial\rho}{\partial t}=\rho^{-1}(hh_t+\sum h_kh_{kt})=\rho^{-1}M(\varepsilon_0 h-\rho^2)\leq \rho^{-1}M(x_0,t)(\varepsilon_0 h-\rho^2).
\end{align*}
According to \cite[Lemma 5.3]{HJ}, it shows that
\begin{align*}
\frac{\partial}{\partial t}|\nabla u|=&-(\nabla^2u)x\cdot \bigg(\frac{\partial h_i}{\partial t}\bigg)e_i-\bigg(\frac{\partial h}{\partial t}\bigg)(\nabla^2 u)x\cdot x-(|\nabla u|^{-1}\nabla u\nabla^2u\cdot x)\bigg(\frac{\partial h}{\partial t}\bigg)-|\nabla u|\bigg(\frac{\partial h}{\partial t}\bigg)\\
=&-(\nabla^2u)x\cdot \bigg(-M_i(h-\epsilon_0)-M h_i\bigg)e_i\\
&+M(h-\epsilon_0)\bigg((\nabla^2 u)x\cdot x+|\nabla u|^{-1}\nabla u\nabla^2u\cdot x+|\nabla u|\bigg)\\
\leq & M(x_0,t)\bigg((\nabla^2u)xh_ie_i+h((\nabla^2 u)x\cdot x+|\nabla u|^{-1}\nabla u\nabla^2u\cdot x+|\nabla u|)\bigg),
\end{align*}
then, for $p<n$ ($p\neq 0$), we obtain
\begin{align}\label{eq513}
&\frac{\partial}{\partial t}(\rho^{n-p}|\nabla u|^{-q})\\
\nonumber\leq&(n-p)\rho^{n-p-2}|\nabla u|^{-q}M(x_0,t)(\varepsilon_0 h-\rho^2)\\
\nonumber&-q|\nabla u|^{-(q+1)}\rho^{n-p}M(x_0,t)\bigg((\nabla^2u)xh_ie_i+h((\nabla^2 u)x\cdot x+|\nabla u|^{-1}\nabla u\nabla^2u\cdot x+|\nabla u|)\bigg).
\end{align}
Thus, from Lemma \ref{lem53} (or Lemma \ref{lem54}) and combining $|\nabla u|$ is bounded with Lemma \ref{lem56}, and dropping some negative terms in (\ref{eq513}), we have
\begin{align*}
\frac{\partial}{\partial t}&(\rho^{n-p}|\nabla u|^{-q})\leq(n-p)\rho^{n-p-2}|\nabla u|^{-q}M(x_0,t)(\varepsilon_0 h-\rho^2)\leq C_1M(x_0,t).
\end{align*}
From (\ref{eq402}) and Lemma \ref{lem41}, we know that
\begin{align*}
\frac{\partial}{\partial t}(\lambda(t))=0.
\end{align*}

We use (\ref{eq207}), (\ref{eq511}) and recall $b_{ij}=\nabla_{ij}h+h\delta_{ij}$ may give
\begin{align*}\frac{\partial (\det(\nabla^2h+hI))^{-1}}{\partial t}=&-(\det(\nabla^2h+hI))^{-2}
\frac{\partial(\det(\nabla^2h+hI))}{\partial b_{ij}}\frac{\partial(\nabla^2h+hI)}{\partial t}\\
=&-(\det(\nabla^2h+hI))^{-2}\frac{\partial(\det(\nabla^2h+hI))}{\partial b_{ij}}
(h_{tij}+h_t\delta_{ij})\\
\leq&(\det(\nabla^2h+hI))^{-2}\frac{\partial(\det(\nabla^2h+hI))}{\partial b_{ij}}
M(b_{ij}-\varepsilon_0\delta_{ij})\\
\leq&\mathcal{K}M[(n-1)-\varepsilon_0(n-1)\mathcal{K}^{\frac{1}{n-1}}].
\end{align*}

Therefore, based on (\ref{eq512}) and the above computations, we have the following conclusion at $(x_0,t)$,
\begin{align}\label{eq514}
\frac{\partial}{\partial t}M\leq\frac{1}{h-\varepsilon_0}\bigg(C_2M^2+f\lambda \rho^{n-p}|\nabla u|^{-q}\mathcal{K}M[(n-1)-\varepsilon_0(n-1)\mathcal{K}^{\frac{1}{n-1}}]\bigg)+M+M^2.
\end{align}
If $M >>1$,
\begin{align*}
\frac{1}{C_3}\mathcal{K}\leq M\leq C_3\mathcal{K}.
\end{align*}
Then, (\ref{eq514}) implies that
\begin{align*}
\frac{\partial}{\partial t}M\leq& \frac{1}{h-\varepsilon_0}\bigg(C_2M^2+f\lambda \rho^{n-p}|\nabla u|^{-q}C_3M^2((n-1)-\varepsilon_0(n-1)(C_3M) ^{\frac{1}{n-1}})\bigg)+M+M^2\\
\leq&\frac{1}{h-\varepsilon_0}M^2\bigg(C_2+[f\lambda \rho^{n-p}|\nabla u|^{-q}C_3(n-1)]-[f\lambda\rho^{n-p}|\nabla u|^{-q}C_3^{\frac{n}{n-1}}(n-1)]\varepsilon_0M ^{\frac{1}{n-1}}+2\bigg)\\
=&\frac{[f\lambda\rho^{n-p}|\nabla u|^{-q}C_3^{\frac{n}{n-1}}(n-1)]}{h-\varepsilon_0}M^2\bigg(\frac{C_2+[f\lambda \rho^{n-p}|\nabla u|^{-q}C_3(n-1)]+2}{[f\lambda\rho^{n-p}|\nabla u|^{-q}C_3^{\frac{n}{n-1}}(n-1)]}-\varepsilon_0M^{\frac{1}{n-1}}\bigg)\\
\leq & C_4M^2(C_5-\varepsilon_0M^{\frac{1}{n-1}})<0.
\end{align*}
Since $C_4$ and $C_5$ depend on $\|f\|_{C^0(S^{n-1})}$, $\|h\|_{C^0(S^{n-1}\times [0,T))}$, $\|h\|_{C^1(S^{n-1}\times [0,T))}$, $\|\lambda\|_{C^0(S^{n-1}\times [0,T))}$ and $|\nabla u|$. Consequently, we get
\begin{align*}
M(x_0,t)\leq C.
\end{align*}
Thus, for any $(x,t)\in S^{n-1}\times [0,T)$, we have
\begin{align*}
\mathcal{K}(x,t)=\frac{(h-\varepsilon_0)M(x,t)+h}{f(x)\rho^{n-p}|\nabla u|^{-q}\lambda}\leq
\frac{(h-\varepsilon_0)M(x_0,t)+h}{f(x)\rho^{n-p}|\nabla u|^{-q}\lambda}\leq C.
\end{align*}

Step 2: Prove $\kappa_i\geq\frac{1}{C}$.

We consider the auxiliary function as follows:
\begin{align*}
E(x,t)=\log\beta_{\max}(\{b_{ij}\})-A\log h+B|\nabla h|^2,
\end{align*}
where $A, B$ are positive constants which will be chosen later, and $\beta_{\max}(\{b_{ij}\})$ denotes the maximal eigenvalue of $\{b_{ij}\}$; for convenience, we write $\{b^{ij}\}$ for $\{b_{ij}\}^{-1}$.

For every fixed $t\in[0,T)$, suppose $\max_{S^{n-1}}E(x,t)$ is attained at point $x_0\in S^{n-1}$. By a rotation of coordinates, we may assume
\begin{align*}
\{b_{ij}(x_0,t)\} \text{ is diagonal,} \quad \text {and} \quad \beta_{\max}(\{b_{ij}\}(x_0,t))=b_{11}(x_0,t).
\end{align*}
Hence, in order to show $\kappa_i\geq\frac{1}{C}$, that is to prove $b_{11}\leq C$. By means of the above assumption, we
transform $E(x,t)$ into the following form,
\begin{align*}
\widetilde{E}(x,t)=\log b_{11}-A\log h+B|\nabla h|^2.
\end{align*}
Utilizing again the above assumption, for any fixed $t \in [0,T)$, $\widetilde{E}(x,t)$ has a local
maximum at $(x_0,t)$, thus, at $(x_0,t)$, we have
\begin{align}\label{eq515}
0=\nabla_i\widetilde{E}=&b^{11}\nabla_ib_{11}-A\frac{h_i}{h}+2B\sum h_kh_{ki}\\
\nonumber=&b^{11}(h_{i11}+h_1\delta_{i1})-A\frac{h_i}{h}+2Bh_ih_{ii},
\end{align}
and
\begin{align*}
0\geq&\nabla_{ii}\widetilde{E}\\
=&\nabla_ib^{11}(h_{i11}+h_1\delta_{i1})+b^{11}
[\nabla_i(h_{i11}+h_1\delta_{i1})]-A\bigg(\frac{h_{ii}}{h}-\frac{h_i^2}{h^2}\bigg)
+2B(\sum h_kh_{kii}+h^2_{ii})\\
=&-(b_{11})^{-2}\nabla_ib_{11}(h_{i11}+h_1\delta_{i1})+b^{11}(\nabla_{ii}b_{11})-A\bigg(\frac{h_{ii}}{h}-\frac{h_i^2}{h^2}\bigg)
+2B(\sum h_kh_{kii}+h^2_{ii})\\
=&b^{11}\nabla_{ii}b_{11}-(b^{11})^2(\nabla_ib_{11})^2-A\bigg(\frac{h_{ii}}{h}-\frac{h_i^2}{h^2}\bigg)
+2B(\sum h_kh_{kii}+h^2_{ii}).	
\end{align*}
In addition, at $(x_0,t)$, we also obtain
\begin{align*}
\frac{\partial}{\partial t}\widetilde{E}=&\frac{1}{b_{11}}\frac{\partial b_{11}}{\partial t}-A\frac{h_t}{h}+2B\sum h_kh_{kt}\\ =&b^{11}\frac{\partial}{\partial t}(h_{11}+h\delta_{11})-A\frac{h_t}{h}+2B\sum h_kh_{kt}\\
=&b^{11}(h_{11t}+h_t)-A\frac{h_t}{h}+2B\sum h_kh_{kt}.
\end{align*}

From Equation (\ref{eq403}) and (\ref{eq207}), we know that
\begin{align}\label{eq416}
\nonumber \log(h-h_t)=&\log(h+\lambda\rho^{n-p}f|\nabla u|^{-q}\mathcal{K}-h)\\
\nonumber=&\log\mathcal{K}+
\log(\lambda\rho^{n-p}f|\nabla u|^{-q})\\
=&-\log[\det(\nabla^2h+hI)]+\log(\lambda\rho^{n-p}f|\nabla u|^{-q}).
\end{align}
Let
\begin{align*}
\mathcal{R}(x,t)=\log(\lambda\rho^{n-p}f|\nabla u|^{-q}).
\end{align*}
Differentiating (\ref{eq416}) once and twice, we respectively get
\begin{align*}\frac{h_k-h_{kt}}{h-h_t}=&-\sum b^{ij}\nabla_kb_{ij}+\nabla_k\mathcal{R}\\
	=&-\sum b^{ii}(h_{kii}+h_i\delta_{ik})+\nabla_k\mathcal{R},
\end{align*}
and
\begin{align*}
\frac{h_{11}-h_{11t}}{h-h_t}-\frac{(h_1-h_{1t})^2}{(h-h_t)^2}=&-\bigg(-\sum (b^{ii})^2(\nabla_{i}b_{ii})^2+b^{ii}\nabla_{ii}b_{ii}\bigg)+\nabla_{11}\mathcal{R}\\
=&-\sum b^{ii}\nabla_{11}b_{ii}+\sum b^{ii}b^{jj}(\nabla_1b_{ij})^2+\nabla_{11}\mathcal{R}.
\end{align*}
By the Ricci identity, we have
\begin{align*}
\nabla_{11}b_{ii}=\nabla_{ii}b_{11}-b_{11}+b_{ii}.
\end{align*}
Thus, we can derive
\begin{align*}
\frac{\frac{\partial}{\partial t}\widetilde{E}}{h-h_t}=&b^{11}\bigg(\frac{h_{11t}+h_t}
	{h-h_t}\bigg)-A\frac{h_t}{h(h-h_t)}
	+\frac{2B\sum h_kh_{kt}}{h-h_t}\\
	=&b^{11}\bigg(\frac{h_{11t}-h_{11}}{h-h_t}+\frac{h_{11}+h-h+h_t}{h-h_t}\bigg)-A\frac{1}{h}
	\frac{h_t-h+h}{h-h_t}
	+\frac{2B\sum h_kh_{kt}}{h-h_t}\\
	=&b^{11}\bigg(-\frac{(h_1-h_{1t})^2}{(h-h_t)^2}+\sum b^{ii}
	\nabla_{11}b_{ii}-\sum b^{ii}b^{jj}(\nabla_1b_{ij})^2-\nabla_{11}\mathcal{R}\bigg.\\
	&\bigg.+\frac{h_{11}+h-{(h-h_t)}}
	{h-h_t}\bigg) -\frac{A}{h}\bigg(\frac{-(h-h_t)+h}{h-h_t}\bigg)+\frac{2B\sum h_kh_{kt}}{h-h_t}\\
	=&b^{11}\bigg(-\frac{(h_1-h_{1t})^2}{(h-h_t)^2}+\sum b^{ii}
	\nabla_{11}b_{ii}-\sum b^{ii}b^{jj}(\nabla_1b_{ij})^2-\nabla_{11}\mathcal{R}\bigg)\\
	&+b^{11}\bigg(\frac{h_{11}+h}{h-h_t}-1\bigg) +\frac{A}{h}-\frac{A}{h-h_t}+\frac{2B\sum h_kh_{kt}}{h-h_t}\\
	=&b^{11}\bigg(-\frac{(h_1-h_{1t})^2}{(h-h_t)^2}+\sum b^{ii}
	\nabla_{11}b_{ii}-\sum b^{ii}b^{jj}(\nabla_1b_{ij})^2-\nabla_{11}\mathcal{R}\bigg)+\frac{1-A}{h-h_t}\\
	&-b^{11}+\frac{A}{h}+\frac{2B\sum h_kh_{kt}}{h-h_t}\\
    \leq&b^{11}\bigg(\sum b^{ii}(\nabla_{ii}b_{11}-b_{11}+b_{ii})-\sum b^{ii}b^{jj}(\nabla_1b_{ij})^2\bigg)
	-b^{11}\nabla_{11}\mathcal{R}+\frac{1-A}{h-h_t}\\
	&+\frac{A}{h}+\frac{2B\sum h_kh_{kt}}{h-h_t}\\
	\leq&\sum b^{ii}\bigg[(b^{11})^2(\nabla_ib_{11})^2+A\bigg(\frac{h_{ii}}{h}-\frac{h_i^2}{h^2}
	\bigg)-2B(\sum h_kh_{kii}+h_{ii}^2)\bigg]\\
	&-b^{11}\sum b^{ii}b^{jj}(\nabla_1b_{ij})^2-b^{11}\nabla_{11}\mathcal{R}+\frac{1-A}{h-h_t}+\frac{A}{h}+\frac{2B\sum h_kh_{kt}}{h-h_t}\\
	\leq&\sum b^{ii}\bigg[A\bigg(\frac{h_{ii}+h-h}{h}-\frac{h_i^2}{h^2}\bigg)\bigg]+2B
	\sum h_k\bigg(-\sum b^{ii}h_{kii}+\frac{h_{kt}}{h-h_t}\bigg)\\
	&-2B\sum b^{ii}(b_{ii}-h)^2-b^{11}\nabla_{11}\mathcal{R}+\frac{1-A}{h-h_t}+\frac{A}{h}\\
	\leq&\sum b^{ii}\bigg[A\bigg(\frac{b_{ii}}{h}-1\bigg)\bigg]+2B\sum h_k\bigg(\frac{h_k}{h-h_t}
	+b^{kk}h_k-\nabla_k\mathcal{R}\bigg)\\
	&-2B\sum b^{ii}(b_{ii}^2-2b_{ii}h)-b^{11}\nabla_{11}\mathcal{R}+\frac{1-A}{h-h_t}+\frac{A}{h}\\
	\leq&-2B\sum h_k\nabla_k\mathcal{R}-b^{11}\nabla_{11}\mathcal{R}+(2B|\nabla h|-A)\sum b^{ii}-2B\sum b_{ii}\\
	&+4B(n-1)h+\frac{2B|\nabla h|^2+1-A}{h-h_t}+\frac{nA}{h}.
\end{align*}

Recall
\begin{align*}
\mathcal{R}(x,t)=\log(\lambda\rho^{n-p}f|\nabla u|^{-q})=\log \lambda+(n-p)\log \rho+\log f-q\log|\nabla u|,
\end{align*}
since $\lambda$ is a constant factor, we have $\lambda_k=0$.
Consequently, we may obtain following conclusion from $\mathcal{R}(x,t)$ and (\ref{eq515}),
\begin{align*}
&-2B\sum h_k\nabla_k\mathcal{R}-b^{11}\nabla_{11}\mathcal{R}\\
=&-2B\sum h_k\bigg(\frac{f_k}{f}+(n-p)\frac{\rho_k}{\rho}-q\frac{(|\nabla u|)_k}{|\nabla u|}
\bigg)-b^{11}\nabla_{11}\mathcal{R}\\
=&-2B\sum h_k\bigg(\frac{f_k}{f}+(n-p)\frac{\rho_k}{\rho}-q\frac{(|\nabla u|)_k}{|\nabla u|}\bigg)\\
&-b^{11}\bigg(\frac{ff_{11}-f_1^2}{f^2}+(n-p)\frac{\rho\rho_{11}-\rho_1^2}{\rho^2}-q\frac{|\nabla u|(|\nabla u|)_{11}-(|\nabla u|)^2_1}{(|\nabla u|)^2}\bigg)\\
\leq&C_6B+C_{7}b^{11}+2qB\sum h_k\frac{(|\nabla u|)_k}{|\nabla u|}-(n-p)b^{11}\frac{\rho\rho_{11}-\rho_1^2}{\rho^2}\\
&+qb^{11}\frac{|\nabla u|(|\nabla u|)_{11}-(|\nabla u|)^2_1}{(|\nabla u|)^2}.
\end{align*}
From $\rho=(h^2+|\nabla h|^2)^{\frac{1}{2}}$, we have
\begin{align*}
\rho_k=\rho^{-1}(hh_k+\Sigma h_kh_{kk})=\rho^{-1}(hh_k+\Sigma h_k(b_{kk}-h\delta_{kk})),
\end{align*}
then,
\begin{align*}
\rho_{11}=\frac{hh_{11}+h_1^2+\Sigma h_1h_{111}+\Sigma h_{11}^2}{\rho}-\frac{h_1^2b_{11}^2}{\rho^3}.
\end{align*}
This and (\ref{eq515}) imply
\begin{align*}
\rho_{11}=\frac{h(b_{11}-h)+h_1^2+\Sigma h_1\bigg(A\frac{h_1}{h}-2Bh(b_{11}-h)-b^{11}(h_1\delta_{11})\bigg )b_{11}}{\rho}-\frac{h_1^2b_{11}^2}{\rho^3}.
\end{align*}
Thus,
\begin{align*}
(n-p)b^{11}\frac{\rho\rho_{11}-\rho_1^2}{\rho^2}\leq C_8b_{11}.
\end{align*}

Recall that
\begin{align*}|\nabla u(X,t)|=-\nabla u(X,t)\cdot x,
\end{align*}
taking the covariant derivative of above equality, we get
\begin{align*}
(|\nabla u|)_k=-b_{ik}((\nabla^2u)e_i\cdot x),
\end{align*}
further,
\begin{align*}
(|\nabla u|)_{11}=&-b_{i11}((\nabla^2u)e_i\cdot x)-b_{j1}b_{i1}((\nabla^3 u)e_je_i\cdot x)\\
&+b_{i1}((\nabla^2u)x\cdot x)-b_{i1}((\nabla^2u)e_i\cdot e_1).
\end{align*}
Thus, combining $|\nabla u|$ is bounded with Lemma \ref{lem56}, we get
\begin{align*}
2qB\sum h_k\frac{(|\nabla u|)_k}{|\nabla u|}=2qB\sum h_k\frac{-b_{ik}((\nabla^2u)e_i\cdot x)}{|\nabla u|}\leq C_{9}Bb_{11}.
\end{align*}
From (\ref{eq515}), we obtain
\begin{align*}
b^{11}b_{i11}=A\frac{h_i}{h}+2Bh_ih_{ii}=A\frac{h_i}{h}+2Bh_i(b_{ii}-h\delta_{ii}).
\end{align*}
Since $|\nabla u|$ and $|\nabla^k u|$ are bounded, and combining Lemma \ref{lem56}, therefore, we get
\begin{align*}
qb^{11}\frac{|\nabla u|(|\nabla u|)_{11}-(|\nabla u|)^2_1}{(|\nabla u|)^2}\leq C_{10}Bb_{11}.
\end{align*}

It follows that
\begin{align*}
\frac{\frac{\partial}{\partial t}\widetilde{E}}{h-h_t}\leq C_{11}Bb_{11}+C_{12}b^{11}+
(2B|\nabla h|-A)\sum b^{ii}-2B\sum b_{ii}+4B(n-1)h+\frac{nA}{h}<0,
\end{align*}
provided $b_{11}>>1$ and if we choose $A>>B$. We obtain
\begin{align*}
\widetilde{E}(x_0,t)\leq C.
\end{align*}
Consequently,
\begin{align*}
E(x_0,t)=\widetilde{E}(x_0,t)\leq C.
\end{align*}
This tells us the principal radii is bounded from above, or equivalently $\kappa_i\geq\frac{1}{C}$.
\end{proof}

\vskip 0pt
\section{\bf The convergence of the flow}\label{sec6}

With the help of priori estimates in the section \ref{sec5}, the long-time existence and asymptotic behaviour of flow (\ref{eq109}) (or (\ref{eq401})) are obtained, we also complete the proof of Theorem \ref{thm13}.
\begin{proof}[Proof of Theorem \ref{thm13}] Since Equation (\ref{eq403}) is parabolic, we can get its short time existence. Let $T$ be the maximal time such that $h(\cdot, t)$ is a smooth solution
to Equation (\ref{eq403}) for all $t\in[0,T)$. Lemma \ref{lem53}-\ref{lem56} enable us to apply Lemma \ref{lem57} to Equation (\ref{eq403}), thus, we can deduce an uniformly upper bound and an uniformly lower bound for the biggest eigenvalue of $\{(h_{ij}+h\delta_{ij})(x,t)\}$. This implies
\begin{align*}
	C^{-1}I\leq (h_{ij}+h\delta_{ij})(x,t)\leq CI,\quad \forall (x,t)\in S^{n-1}\times [0,T),
\end{align*}
where $C>0$ is independent of $t$. This shows that Equation (\ref{eq403}) is uniformly parabolic. Estimates for higher derivatives follow from the standard regularity theory of uniformly parabolic equations
Krylov \cite{KR}. Hence, we obtain the long time existence and regularity of solutions for the flow
(\ref{eq109}) (or (\ref{eq401})). Moreover, we obtain
\begin{align}\label{eq601}
\|h\|_{C^{l,m}_{x,t}(S^{n-1}\times [0,T))}\leq C_{l,m},
\end{align}
where $C_{l,m}$ ($l, m$ are nonnegative integers pairs) are independent of $t$, then $T=\infty$. Using the parabolic comparison principle, we can attain the uniqueness of smooth even solutions $h(\cdot,t)$ of Equation (\ref{eq403}) for $p<n$ ($p\neq 0$). Specially, we also get the uniqueness of smooth non-even solutions $h(\cdot,t)$ with $p<0$.

Now, we recall the non-increasing property of $\Phi(\Omega_t)$ in Lemma \ref{lem42}, we know that

\begin{align}\label{eq602}
\frac{\partial\Phi(\Omega_t)}{\partial t}\leq 0.
\end{align}
Based on (\ref{eq602}), there exists a $t_0$ such that
\begin{align*}
\frac{\partial\Phi(\Omega_t)}{\partial t}\bigg|_{t=t_0}=0,
\end{align*}
this yields
\begin{align*}
|\nabla u|^qh\rho^{p-n}\det(\nabla_{ij}h+h\delta_{ij})=\lambda(t_0)f,
\end{align*}
equivalently,
\begin{align*}
|\nabla u|^qh(h^2+|\nabla h|^2)^{\frac{p-n}{2}}\det(\nabla_{ij}h+h\delta_{ij})=\lambda(t_0)f.
\end{align*}
Let $\Omega=\Omega_{t_0}$ and $\lambda(t_0)=\lambda$. Thus, the support function of $\Omega$ satisfies (\ref{eq108}).

In view of (\ref{eq601}), applying the Arzel$\grave{a}$-Ascoli theorem and a diagonal argument, we can extract
a subsequence of $t$, it is denoted by $\{t_j\}_{j \in \mathbb{N}}\subset (0,+\infty)$, and there exists a smooth function $h(x)$ such that
\begin{align}\label{eq603}
\|h(x,t_j)-h(x)\|_{C^i(S^{n-1})}\rightarrow 0,
\end{align}
uniformly for each nonnegative integer $i$ as $t_j \rightarrow \infty$. This reveals that $h(x)$ is a support function. Let us denote by $\Omega$ the convex body determined by $h(x)$. Thus, $\Omega$ is a smooth, origin-symmetric strictly convex body for $p<n$ $(p\neq0)$. Specially, for $p<0$, $\Omega$ is a smooth strictly convex body containing the origin in its interior.

Moreover, by (\ref{eq601}) and the uniform estimates in section \ref{sec5}, we conclude that $\Phi(\Omega_t)$ is a bounded function in $t$ and $\frac{\partial \Phi(\Omega_t)}{\partial t}$ is uniformly continuous. Thus, for any $t>0$, by the monotonicity of $\Phi$ in Lemma \ref{lem42}, there is a constant $C>0$ being independent of $t$, such that
\begin{align*}
\int_0^t\bigg(-\frac{\partial\Phi(\Omega_t)}{\partial t}\bigg)dt=\Phi(\Omega_0)-\Phi(\Omega_t)\leq C,
\end{align*}
this gives
\begin{align}\label{eq604}
\lim_{t\rightarrow\infty}\Phi(\Omega_t)-\Phi(\Omega_0)=
-\int_0^\infty\bigg|\frac{\partial}{\partial t}\Phi(\Omega_t)\bigg|dt\leq C.
\end{align}
The left hand side of (\ref{eq604}) is bounded below by $-2C$, therefore, there is a subsequence $t_j\rightarrow\infty$ such that
\begin{align*}
\frac{\partial}{\partial t}\Phi(\Omega_{t_j})\rightarrow 0 \quad\text{as}\quad  t_j\rightarrow\infty.
\end{align*}
The proof of Lemma \ref{lem42} shows that
\begin{align}\label{eq605}
\frac{\partial\Phi(\Omega_t)}{\partial t}\bigg|_{t=t_j}=-\int_{S^{n-1}}\frac{[-\lambda(t)\frac{f(x)\rho^{n-p}\mathcal{K}}{|\nabla u|^q}+h]^2}{\frac{h\rho^{n-p}\mathcal{K}}{|\nabla u|^q}\int_{S^{n-1}}\rho^p|\nabla u|^qdv}dx\bigg|_{t=t_j}\leq 0.
\end{align}
Taking the limit $t_j\rightarrow\infty$, the equality condition of (\ref{eq605}) means that there has
\begin{align*}
|\nabla u(X^\infty)|^qh^{\infty}[(h^\infty)^2+|\nabla (h^\infty)|^2]^{\frac{p-n}{2}}\det(\nabla_{ij}h^{\infty}+h^{\infty}\delta_{ij})=\frac{\frac{a}{b}\widetilde{Q}_{q,n-p}(\Omega_\infty)}{\int_{S^{n-1}}f(x)dx}f(x),
\end{align*}
which satisfies (\ref{eq108}), where $h^{\infty}$ is the support function of $\Omega^{\infty}$, and $X^{\infty}=\nabla h^{\infty}$.
This provides the proof of Theorem \ref{thm13}.
\end{proof}

\end{document}